\documentclass[11pt,twoside]{article}
\usepackage[utf8]{inputenc}
\usepackage[margin=1in]{geometry}
\usepackage[T1]{fontenc}
\usepackage{amsmath,amsthm,amssymb}
\usepackage{newcent,microtype}
\usepackage[hidelinks]{hyperref}
\usepackage{rotating}

\swapnumbers
\newtheorem{theorem}{Theorem}[section]
\newtheorem{assumption}[theorem]{Assumption}
\newtheorem{corollary}[theorem]{Corollary}
\newtheorem{definition}[theorem]{Definition}
\newtheorem{lemma}[theorem]{Lemma}
\newtheorem{proposition}[theorem]{Proposition}

\DeclareMathOperator\ad{ad}

\DeclareMathOperator\Aut{Aut}

\DeclareMathOperator\End{End}

\DeclareMathOperator\Irr{Irr}

\DeclareMathOperator\Spec{Spec}

\DeclareMathOperator\Sym{Sym}

\newcommand\la{\langle}

\newcommand\normal{\triangleleft}

\newcommand\ra{\rangle}

\newcommand\size[1]{\lvert #1\rvert}
\newcommand\tensor{\otimes}

\newcommand\CC{\mathbb{C}}

\newcommand\kk{\mathbf{k}}
\newcommand\NN{\mathbb{N}}
\newcommand\QQ{\mathbb{Q}}
\newcommand\RR{\mathbb{R}}
\newcommand\ZZ{\mathbb{Z}}

\newcommand\op{\mathrm}
\renewcommand\cal{\mathcal}
\renewcommand\frak{\mathfrak}

\newcommand\al{\alpha}
\newcommand\bt{\beta}
\newcommand\gm{\gamma}		
\newcommand\dl{\delta}

\newcommand\lm{\lambda}	
\newcommand\sg{\sigma}

\begin{document}
\abovedisplayskip=0.5em
\belowdisplayskip=0.5em

\title{Universal Axial Algebras and a Theorem of Sakuma}
\author{J. I. Hall, F. Rehren, S. Shpectorov}
\date{\today}
\maketitle

\begin{center}
    {\em Dedicated to the  memory of  \'Akos Seress}
\end{center}

\begin{abstract}
	In the first half of this paper,
	we define axial algebras:
	nonassociative commutative algebras generated by axes,
	that is, semisimple idempotents---%
	the prototypical example of which is Griess' algebra \cite{conway}
	for the Monster group.
	When multiplication of eigenspaces of axes is controlled by fusion rules,
	the structure of the axial algebra is determined to a large degree.
	We give a construction of the universal Frobenius axial algebra on $n$ generators
	with a specified fusion rules,
	of which all $n$-generated Frobenius axial algebras with the same fusion rules
	are quotients.
	In the second half,
	we realise this construction in the Majorana / Ising / $\op{Vir}(4,3)$-case
	on $2$ generators,
	and deduce a result generalising Sakuma's theorem in VOAs \cite{sakuma}.
\end{abstract}

\section{Introduction}
\label{sec-intro}
Very often after an important theorem is first proven,
there is fruitful work in refining the hypothesis and techniques
until we feel the theorem has found its proper place.
In this case, our point of departure
is an observation of Miyamoto \cite{miyamoto-s3}
and a theorem of Sakuma \cite{sakuma}
in the case of moonshine-type Vertex Operator Algebras.
In general, a Vertex Operator Algebra (VOA) is a direct sum
of infinitely many finite-dimensional vector spaces,
graded by weight,
together with a somewhat subtle product structure.
A moonshine-type VOA $V$
has a finite-dimensional vector space $V_2$ at weight $2$
which admits a nonassociative commutative algebra structure
and an associative bilinear form.
Miyamoto showed that certain semisimple idempotents in $V_2$,
analogous to our axes,
induce finite automorphisms on the entire VOA.
Moreover, the axes are locally identifiable
and carry a representation of a Virasoro algebra \cite{miyamoto}.
Sakuma then showed that the product of two such automorphisms
has finite order bounded above by $6$.
This insight is an evolution of the construction of the Monster group,
the largest sporadic simple group;
it was first realised as the automorphism group of the Griess algebra $\cal G$
by Griess \cite{griess},
and then as the automorphism group of the moonshine VOA $V^\natural$,
whose weight-$2$ component is indeed the Griess algebra $\cal G$,
in the pioneering works of Borcherds \cite{borcherds}
and Frenkel, Lepowsky and Meurman \cite{flm}.
In fact, $\cal G$ is an axial algebra in our sense.
An early analysis of the axes in $\cal G$
was carried out by Norton and recorded, {\em inter alia},
by Conway in \cite{conway}.

In 2009, Ivanov axiomatised a new class of algebras \cite{ivanov},
called Majorana algebras,
with strikingly similar behaviour to the weight-$2$ component
of moonshine-type VOAs.
In particular, Sakuma's theorem still holds
and moreover the isomorphism type of an algebra generated by two axes
is restricted by the order of the product of their automorphisms \cite{ipss}.
Most spectacularly,
this captures the Griess algebra
in an entire family of algebras.
This has opened a new horizon on studying the Monster,
and other groups,
via a new kind of representation.

We propose a larger class of algebras as a natural object of study.
These `axial algebras' can be seen as a generalisation of commutative associative algebras,
as well as capturing Majorana algebras,
the algebras occurring at weight-$2$ in moonshine-type VOAs,
some Jordan algebras,
and perhaps other algebras relating to mathematical physics.

In this article, we study axial algebras in some detail.
We now give an overview of the work presented here.
The extent to which our algebras generalise associativity 
is controlled by {\em fusion rules},
which we define in Section \ref{sec-fus}.
If structure constants describe the multiplication of elements in an algebra,
fusion rules can describe the multiplication of submodules of an algebra.
A particularly rich source of fusion rules $\frak V(p,q)$ is
the `discrete series $\op{Vir}(p,q)$ of Virasoro algebras',
and we also discuss this case;
it happens to arise in many situations of mathematical physics,
such as the Ising and $n$-state Potts models of statistical mechanics \cite{cft},
and in particular in VOAs.

A semisimple element $a$ of an algebra
affords a decomposition of the algebra into eigenspaces.
If these eigenspaces obey the fusion rules $\frak F$,
and furthermore $a$ is an idempotent,
we call $a$ an {\em $\frak F$-axis}.
If the algebra is associative,
then $a$ can only have eigenvalues $0$ and $1$,
and the fusion rules are likewise very restricted.
The nonassociative case is much richer.
We call a nonassociative commutative algebra an {\em $\frak F$-axial algebra}
when it is generated by $\frak F$-axes.
If furthermore the algebra possesses an associating bilinear form,
we speak of a {\em Frobenius} $\frak F$-axial algebra.
These concepts are introduced in Section \ref{sec-ax}.

A key point of interest for us is that
an $\frak F$-axis induces a finite-order automorphism of the entire algebra
if the fusion rules $\frak F$ admit the appropriate finite-order grading.
In particular, if $\frak F$ is $\ZZ/2$-graded,
as in the Virasoro case mentioned above,
then each $\frak F$-axis induces an involution in the automorphism group of the algebra.
Therefore the automorphism groups of such $\frak F$-axial algebras are very rich.
For example, the Griess algebra $\cal G$ contains $\frak V(4,3)$-axes
whose induced automorphisms form precisely the smallest conjugacy class $2A$
of involutions in the Monster,
in its representation as the automorphism group of $\cal G$.
The automorphism groups of axial algebras,
and in particular axial algebras occurring as representations of transposition groups,
will be studied in a later paper.

In the first half of this paper,
we consider $\frak F$-axial algebras as standalone objects in their own right.
The principal result, over rings containing a subfield $\kk$, is
\begin{theorem}
	\label{thm-exists-universal}
	There exists a {\em universal $n$-generated Frobenius $\frak F$-axial algebra}.
	All Frobenius $\frak F$-axial algebras on $n$ generators
	are quotients of this universal object.
\end{theorem}
It is a corollary of Theorem \ref{universal axial},
which gives a construction of the universal Frobenius $\frak F$-axial algebra on $n$ generators.
The proof is in two parts:
Section \ref{sec-uni-alg} has a construction of a universal nonassociative algebra with associative bilinear form,
and in Section \ref{sec-uni-ax} we specialise to the case
where this is generated by $\frak F$-axes.

In the second half of the paper, we specialise further to find
the universal Frobenius $\frak V(4,3)$-axial algebra on $2$ generators over $\QQ$.
Since $\frak V(4,3)$ is a fusion rules with $\ZZ/2$-grading,
any $\frak V(4,3)$-axis induces an involutory automorphism,
and any two involutory automorphisms generate a dihedral subgroup of the automorphism group.
The proof proceeds by some observations
on arbitrary $2$-generated Frobenius $\frak V(4,3)$-axial algebras (Section \ref{sec-two-gen}),
and explicitly computing a multiplication table (Section \ref{sec-cal})
for the universal object.
Finally, in Section \ref{sec-sakuma},
we compute some relations to classify its quotients:
\begin{theorem}
	\label{thm-intro-sakuma}
	There are $9$ Frobenius $\frak V(4,3)$-axial algebras on $2$ generators
	over any field of characteristic $0$.
	They are exactly the Norton-Sakuma algebras of \cite{conway}, \cite{sakuma}.
\end{theorem}

For us, this theorem is a corollary of Theorem \ref{thm-sakuma},
which describes the universal $2$-generated Frobenius $\frak V(4,3)$-axial algebra
as the direct sum of the Norton-Sakuma algebras.
It proves, in a more general setting,
results first obtained by Sakuma and then Ivanov et al.
Namely, in \cite{sakuma}, Sakuma obtained the startling result
that any two $\frak V(4,3)$-axes (in our terminology)
in the weight-$2$ component of a moonshine-type VOA
induce a dihedral subgroup of the automorphism group of size at most $12$.
In \cite{ipss}, Ivanov, Pasechnik, Seress and Shpectorov
proved that, over $\RR$,
the isomorphism type of an algebra with a positive definite associating bilinear form
generated by two $\frak V(4,3)$-axes
must be one of the nine possibilities first observed in \cite{conway}.
Our approach works over a more general ring,
and indeed the form is not assumed to be positive-definite.
Furthermore it can be seen that the algebra products
may be explicitly derived from the axioms of our situation.

The authors would like to thank Alonso Castillo-Ramirez and Miles Reid
for their help and advice,
and Sasha Ivanov, Sophie Decelle and Alonso Castillo-Ramirez
for the organisation of the stimulating {\em Monster, Majorana and Beyond} conference.

\section{Fusion rules}
\label{sec-fus}

Recall that in the representation theory of finite groups over the complex numbers,
a group $G$ has a finite set of irreducible modules $\Irr G = \{M_i\}$,
and every $G$-module is completely reducible.
The tensor product of two $G$-modules is again a $G$-module,
and hence admits a decomposition into irreducible modules.
Here we see the best-known instance of fusion rules,
as a map $f\colon\Irr G\tensor\Irr G\to\ZZ\Irr G$,
which assigns to $M_i\tensor M_j$ a vector $(f_{ij}^k)$
such that $M_i\tensor M_j \cong f_{ij}^1M_1\oplus f_{ij}^2M_2\oplus \dotsm$.
For more details and examples, we refer to \cite{fuchs}.

\begin{definition}
	\label{def-fusion-rules}
	A {\em fusion rules} $\frak F$ over a field $\kk$
	is a triple:
		the {\em central charge} $\op{CC}(\frak F) = c\in\kk-\{0\}$,
		the set $\op F(\frak F)\subseteq\kk$ of {\em fields}, and
		the {\em rules} $\star:\op F(\frak F)\times\op F(\frak F)\to \ZZ\op F(\frak F)$.
\end{definition}
For simplicity of notation, we use $\frak F$
to both refer to the triple $(\op{CC}(\frak F),\op F(\frak F),\star)$
and to the collection of fields $\op F(\frak F)$.
In principle, the fields of $\frak F$
parametrise a collection of irreducible modules,
the rules $\star$ specify the decomposition of their products,
and the central charge $c$ is an additional parameter
whose meaning will be discussed later.

We will only consider fusion rules $\frak F$
satisfying two additional conditions:
\begin{assumption}
	The fusion rules $\frak F$ have finitely many fields, and
	for any fields $f,f',f''\in\frak F$,
		$f''$ has multiplicity at most $1$ in $f\star f'$.
\end{assumption}
We have the latter condition 
because fields occurring with multiplicity greater than $1$
have no interpretation in our setup.
In particular, this implies
that $\star$ can take its values in the power set $\cal P(\frak F)$ of fields.
We will view $\star$ as $\cal P(\frak F)$-valued from now on.

Take $G$ a finite abelian group.
We say that $\frak F$ is $G$-graded
if there exists a partition $\{\frak F_g\}_{g\in G}$ of fields $\frak F$
such that, for all $g,h \in G$,
if $f\in \frak F_g$ and $f'\in\frak F_h$
then $f\star f'\subseteq\frak F_{gh}$.
In the simplest case, if $\frak F$ is $\ZZ/2$-graded,
we write $\ZZ/2$ as $\{+,-\}$ and correspondingly split 
$\frak F$ into $\frak F_+$ and $\frak F_-$.

While we do not give a review of the representation theory of Virasoro algebras,
we wish to summarise some results on the discrete series of Virasoro algebras
to allow the reader to make some calculations of fusion rules of her own.
The remainder of this section can also be skimmed for its examples.

Recall that the {\em Virasoro algebra}
is a central extension of the infinite-dimensional Witt Lie algebra,
with the following presentation:
\begin{equation}
	\op{Vir}_c = \CC\{L_i\}_{i\in\ZZ}\oplus\CC c,\quad
	[c,L_n] = 0,\quad
	[L_m,L_n] = (m-n)L_{m+n} + \frac{c}{12}m(m^2-1)\dl_{m+n}.
\end{equation}
In the action of $\op{Vir}_c$ on a module,
the central element $c$ can only act by a scalar in $\CC$.
By abuse of notation, we identify the central element $c$
with this scalar $c\in\CC$.
The representation theory of $\op{Vir}_c$ depends entirely on the value of $c$.
In particular,
we say that $\op{Vir}_c$ is {\em rational}
if it has finitely many irreducible modules up to isomorphism.
This occurs if and only if \cite{wang} $c\geq1$ or
\begin{equation}
	\label{def-cc-pq}
	c = c(p,q) = 1 - 6(p-q)^2/pq,\quad p,q\in\{2,3,\dotsc\}\text{ coprime}.
\end{equation}
From now on, we may write $\op{Vir}(p,q)$ for $\op{Vir}_{c(p,q)}$,
and {\em whenever we write $c(p,q)$} it is understood that
$p,q\in\{2,3,\dotsc\}$ are coprime integers.

The modules of $\op{Vir}(p,q)$
are parametrised by their highest weights
\begin{equation}
	\label{def-vir-fields}
	h(r,s) = h_{c(p,q)}(r,s) = \frac{(sp-rq)^2 - (p-q)^2}{4pq},\quad
	0 < r < p,\quad 0 < s < q.
\end{equation}
The number of distinct highest weights is linear in $p$ and $q$:
namely, $\frac{1}{2}(p-1)(q-1)$.
Furthermore the decomposition into irreducible modules
of the product of two modules
is governed by admissible triples,
that is,
\begin{equation}
	M_{h(r,s)}\otimes M_{h(t,u)} \leq
	\sum^{\min\{r+t-1,2p-r-t-1\}}_{v = 1 + \size{r-t},\\v \equiv_2 1+r+t}\quad
	\sum^{\min\{s+u-1,2q-s-u-1\}}_{w = 1 + \size{s-u},\\w \equiv_2 1+s+u}
	M_{h(v,w)}.
\end{equation}
It is clear from the above
that the decomposition is symmetric, and
that the multiplicity of any irreducible summand module is at most $1$.
It also follows that the fusion rules admit a unique nontrivial $\ZZ/2$-grading
if $p,q\geq 3$.

To adjust the data from the Virasoro representations
to the way it appears inside algebras,
as in, say, the weight-$2$ component of a moonshine-type VOA,
we have to add an extra field and do some rescaling.
\begin{definition}
A {\em Virasoro fusion rules} $\frak V(p,q)$ has
	central charge $c = c(p,q)$ from \eqref{def-cc-pq},
	fields 
	$$\{1\}\cup\left\{\frac{1}{2}h(r,s)\right\}_{0<r<p,\;0<s<q}
	\quad \text{ from \eqref{def-vir-fields} for }\op{Vir}(p,q),$$
	and fusion rules $\star$ as in \eqref{def-vir-rules} below.
\end{definition}
To define $\star$, temporarily write $\bullet$ for the binary operation
$$ h(r,s)\bullet h(r',s') = \{ h(r'',s'')\mid M_{h(r'',s'')}
	\text{ is direct summand of } M_{h(r,s)}\tensor M_{h(r',s')}\}. $$
Then define the fusion rules $\star$ of $\frak V(p,q)$ by
\begin{equation}
	\label{def-vir-rules}
	\begin{aligned}
	1\star \frac{1}{2}h(r,s) & = \frac{1}{2}h(r,s),\quad 1\star1=1, \\
	\frac{1}{2}h(r,s)\star \frac{1}{2}h(r',s') & = \begin{cases}
		\frac{1}{2}(h(r,s)\bullet h(r',s'))
			& \text{ if } 0\not\in h(r,s)\bullet h(r',s') \\
		\{1\}\cup \frac{1}{2}(h(r,s)\bullet h(r',s'))
			& \text{ if } 0\in h(r,s)\bullet h(r',s')
	\end{cases}
	\end{aligned}
\end{equation}

We give some examples.
The Virasoro algebra $\op{Vir}(4,3)$
has fields $0,\frac{1}{2},\frac{1}{16}$
and central charge $c = c(4,3) = 1 - \frac{6(1)^2}{12} = \frac{1}{2}$.
(It is well-known in the context of the Ising model,
or a fermion VOA \cite{cft}.)
Following our recipe for the fusion rules,
we have that $\frak V(4,3)$ is a fusion rules with
	central charge $\frac{1}{2}$,
	fields $1,0,\frac{1}{4},\frac{1}{32}$, and
	fusion rules
		\renewcommand{\arraystretch}{1.75}
		\begin{center}
		\begin{tabular}{|c||c|c|c|c|}
			\hline
					&$1$ &$0$ & $\frac{1}{4}$ & $\frac{1}{32}$\\
			\hline\hline
			$1$ &$1$ &$0$ & $\frac{1}{4}$ & $\frac{1}{32}$\\
			\hline
			$0$ &$0$ &$1,0$ & $\frac{1}{4}$ & $\frac{1}{32}$\\
			\hline
				$\frac{1}{4}$ & $\frac{1}{4}$ & $\frac{1}{4}$ &$1,0$ & $\frac{1}{32}$\\
			\hline
				$\frac{1}{32}$ & $\frac{1}{32}$ & $\frac{1}{32}$ & $\frac{1}{32}$ &$1,0, \frac{1}{4}$\\
			\hline
		\end{tabular}
		\end{center}
The $\ZZ/2$-grading is given by
$$ \frak V(4,3)_+ = \left\{ 1,0,\frac{1}{4} \right\} \text{ and }
	\frak V(4,3)_- = \left\{\frac{1}{32}\right\}. $$

Here are also the fusion rules $\frak V(5,3)$:
	with central charge $-\frac{3}{5}$,
	fields $1,0,\frac{1}{10},-\frac{1}{40},\frac{3}{8}$, and
	fusion rules
		\renewcommand{\arraystretch}{1.75}
		\begin{center}
		\begin{tabular}{|c||c|c|c|c|c|}
			\hline
					& $1$ & $0$ & $\frac{1}{10}$ & $-\frac{1}{40}$ & $\frac{3}{8}$ \\
			\hline\hline
				$1$ & $1$ & $0$ & $\frac{1}{10}$ & $-\frac{1}{40}$ & $\frac{3}{8}$ \\
			\hline
				$0$ & $0$ & $1,0$ & $\frac{1}{10}$ & $-\frac{1}{40}$ & $\frac{3}{8}$ \\
			\hline
				$\frac{1}{10}$ & $\frac{1}{10}$ & $\frac{1}{10}$ & $1, 0, \frac{1}{10}$ & $-\frac{1}{40},\frac{3}{8}$ & $-\frac{1}{40}$ \\
			\hline
				$-\frac{1}{40}$ & $-\frac{1}{40}$ & $-\frac{1}{40}$ & $-\frac{1}{40},\frac{3}{8}$ & $1$, $0$, $\frac{1}{10}$ & $\frac{1}{10}$ \\
			\hline
				$\frac{3}{8}$  & $\frac{3}{8}$  & $\frac{3}{8}$  & $-\frac{1}{40}$ & $\frac{1}{10}$ & $1, 0$ \\
			\hline
		\end{tabular}
		\end{center}
The $\ZZ/2$-grading is
$$ \frak V(5,3)_+ = \left\{ 1, 0, \frac{1}{10} \right\} \text{ and }
	\frak V(5,3)_- = \left\{ -\frac{1}{40}, \frac{3}{8} \right\}. $$

\section{Axial algebras}
\label{sec-ax}
Throughout, we let $R$ be a ring with $1$
containing a subfield $\kk$.
By {\em algebra} we mean a not-necessarily-associative commutative $R$-module $A$
with $R$-linear binary product.
Recall that $a\in A$ is an idempotent if $aa = a$.
For $a\in A$, define its adjoint $\ad(a)\in\End(A)$
as the mapping by `side'-multiplication $x\mapsto ax$
(right- and left-multiplication being the same).
We say that $a$ is {\em semisimple} if its adjoint matrix $\ad(a)$
is diagonalisable.
Equivalently, this means that $A$ has a direct sum decomposition
into $\ad(a)$-eigenspaces.
Also, we write $\Spec(a) = \Spec(\ad(a))$
for the set of eigenvalues of $\ad(a)$.
With respect to $a$
we furthermore write $A_\lm = \{x\in A\mid \ad(a)x = ax = \lm x\}$
for the $\lm$-eigenspace of $\ad(a)$
(and if $\ad(a)$ has no $\lm$-eigenvectors, we set $A_\lm = 0$),
by convention $A_{\{\lm,\mu\}} = A_\lm + A_\mu$,
and hence finally $A = A_{\Spec(a)}$.

Since our axial algebras somehow generalise associative algebras,
we start by considering the latter.
Suppose $A$ is associative and $a\in A$ is a semisimple idempotent.
If $x$ is a $\lm$-eigenvector of $\ad(a)$, say,
then $ax = \lm x$.
But also $\lm x = ax = (aa)x = a(ax) = a(\lm x) = \lm ax = \lm^2 x$,
and indeed we have that $\lm^n = \lm$ for all $n\in\NN$.
Thus $\lm$ can only be $0$ or $1$,
and $\ad(a)$ has only $0$- or $1$-eigenvectors.
In other words, $A$ has a direct sum decomposition
	$$ A = \op{Fix}(\ad(a))\oplus\op{Ann}(\ad(a)) = A_1 \oplus A_0. $$
Observe that, for $u,v\in A_1$ and $x,y\in A_0$,
	$$ a(uv) = (au)v = uv,\quad ux = (au)x = a(ux) = a(xu) = (ax)u = 0u = 0,\quad a(xy) = (ax)y = 0. $$
Therefore products of eigenvectors are again eigenvectors,
and we can speak of the following rules for the products of eigenvectors:
	\renewcommand{\arraystretch}{1.75}
	\begin{center}
	\begin{tabular}{|c||c|c|}
		\hline
				& $1$ & $0$ \\
		\hline\hline
			$1$ & $1$ & $0$ \\
		\hline
			$0$ & $0$ & $0$ \\
		\hline
	\end{tabular}
	\end{center}
Conversely, an algebra generated by semisimple idempotents
whose $1$-eigenspaces are $1$-dimensional
and whose eigenvectors satisfy these rules
is necessarily associative.
An axial algebra generalises associative algebras
by dropping the requirement of associativity
but axiomatising product (fusion) rules for eigenvectors of semisimple idempotents.

Suppose that $A$ is an algebra 
and $\frak F$ an arbitrary fusion rules as in Definition \ref{def-fusion-rules}.
\begin{definition}
	\label{def-axis}
	An element $a\in A$ is an $\frak F$-axis if 
	\begin{enumerate}
		\item[(a)] $a$ is an idempotent;
		\item[(b)] $a$ is semisimple, 
			$\Spec(a)\subseteq\frak F$ and it spans its own $1$-eigenspace,
			that is, $A_1 = \la a\ra$;
		\item[(c)] the {\em $\frak F$-rules} hold, that is,
			the eigenspaces $A_f,A_g$ for $f,g\in\frak F$
			satisfy $A_fA_g\subseteq A_{f\star g}$
			under the algebra product.
	\end{enumerate}
\end{definition}

If an idempotent $a\in A$ spans its $1$-eigenspace $A_1$,
as we demand in (b) above,
it will be called {\em primitive}
(in the literature, this is often {\em absolute primitivity}).
In particular, a primitive idempotent
does not admit a decomposition into a sum of pairwise annihilating idempotents.

\begin{definition}
	\label{def-ax-alg}
	A nonassociative commutative algebra $A$
	is an {\em $\frak F$-axial algebra}
	if it is generated by $\frak F$-axes.
\end{definition}

\begin{definition}
	\label{def-frobenius}
	An $\frak F$-axial algebra $A$ is {\em Frobenius}
	if it admits an associating bilinear form $\la,\ra$
	(that is, for $x,y,z\in A$, $\la xy,z\ra = \la x,yz\ra$)
	and, for any $\frak F$-axis $a$,
	$\la a,a\ra = 2\op{CC}(\frak F)$.
\end{definition}

Hence, for Frobenius $\frak F$-axial algebras,
we extend condition (a) to
\begin{enumerate}
	\item[(a')] $a$ is an idempotent and $\la a,a\ra = 2\op{CC}(\frak F)$.
\end{enumerate}
We can see that in the definition of a general $\frak F$-axial algebra,
the central charge of $\frak F$ plays no role.
In the specialised Frobenius case,
we use the central charge to fix the scaling of the form.
This is related to applications of our theory:
for example, an $\frak F$-axial algebra occuring in a VOA
has exactly the scaling fixed in this condition (a').
On the other hand, from the point of view of the algebra,
the form can be arbitrarily rescaled,
so long as it is not zero on the axes.

\begin{proposition}
	\label{prop-miyamoto-inv}
	If $a\in A$ is an $\frak F$-axis
    and $\frak F$ is $\ZZ/2$-graded into $\frak F_+$ and $\frak F_-$,
	then $a\in A$ produces an algebra automorphism,
	the {\em Miyamoto involution $\tau(a)$}
		\begin{equation}
			x^{\tau(a)} = \begin{cases}
				x & \text{ if } x \in A_{\frak F_+} \\
				-x & \text{ if } x \in A_{\frak F_-}, \end{cases}
		\end{equation}
	of order at most $2$ on $A$.
	\qed
\end{proposition}

We give an example of an axial algebra.
Let $A = \QQ\{a,b,c\}$ be a $3$-dimensional rational vector space
with algebra product given by, for any $\{x,y,z\} = \{a,b,c\}$,
\begin{equation}
	xx = x,\quad xy = \frac{1}{64}(x + y - z).
\end{equation}
Then the adjoint matrix of $x$ may be computed to deduce that,
with respect to $x$, $A$ decomposes into eigenspaces
\begin{equation}
	A_1 = \la x\ra,\quad
	A_0 = \la x - 32(y + z) \ra,\quad
	A_\frac{1}{32} = \la y - z \ra.
\end{equation}
Moreover it can be calculated that
this decomposition satisfies the fusion rules of $\frak V(4,3)$,
although the field $\frac{1}{4}$ does not occur.
Hence by Proposition \ref{prop-miyamoto-inv} it follows that $\tau(x)\in\Aut(A)$
is a nontrivial automorphism sending $y-z$ to $z-y$,
and hence $\tau(x)$ swaps $y$ and $z$.
The presentation of the group now obviously implies that
$\la\tau(a),\tau(b),\tau(c)\ra \cong Sym(3)$,
which turns out to be the full automorphism group.
Finally we may define a bilinear form by $\la x,y\ra = \frac{1}{64}$
for $x\neq y\in\{a,b,c\}$
and fixing $\la x,x\ra = 1 = 2\op{CC}(\frak V(4,3))$ for all $x\in\{a,b,c\}$;
by calculation this form is associating.
Thus $A$ is a Frobenius $\frak V(4,3)$-axial algebra.
In the literature, $A$ is usually referred to by its isomorphism type $(3C)$.

In fact, all $3$-transposition groups
occur as automorphism groups of $\frak F$-axial algebras
where $\frak F$ is a fusion rules with three fields $\{1,0,\theta\}$;
indeed, the above construction gives a Frobenius axial algebra
if $\frac{1}{32}$ is replaced by an arbitrary field $\theta\in\RR$,
and the construction generalises to put an algebra structure
on the Fischer space of a $3$-transposition group.
For more details, see \cite{hrs}.

Frobenius $\frak V(4,3)$-axial algebras have also been studied
with some additional assumptions
as Majorana algebras.
See for example \cite{ipss},
which first proved Sakuma's theorem in a more general setting.

Finally, we record some simple results.

\begin{proposition}
	\label{lem-form-sym}
	A associative bilinear form $\la,\ra$ on an $\frak F$-axial algebra $A$ is symmetric.
\end{proposition}
\begin{proof}
	Since $A$ is generated by $\frak F$-axes,
	it is generated by idempotents.
	Furthermore $A$ is spanned by monomials,
	each of which is therefore a product of two other elements.
	Suppose $b,c$ are arbitrary monomial products of axes.
	Then $b = b_1b_2$ and
	$ \la b,c\ra = \la b_1b_2,c \ra = \la b_1,b_2c \ra = \la b_1c,b_2\ra = \la c,b_1b_2\ra = \la c,b\ra $,
	as required.
\end{proof}

\begin{proposition}
	\label{lem-form-perp}
	The eigenspaces of an $\frak F$-axis in an algebra $A$
	are perpendicular with respect to an associative bilinear form $\la,\ra$ on $A$.
\end{proposition}
\begin{proof}
	Suppose $x\in A_\lm$ and $y\in A_\mu$ with respect to $a$.
	Then $\lm\la x,y\ra = \la \lm x,y\ra = \la ax,y\ra = \la x,ay\ra = \la x,\mu y\ra = \mu\la x,y\ra$.
	It follows that $(\lm-\mu)\la x,y\ra = 0$.
	But the fields $\lm,\mu\in\frak F$ lie in the field $\kk$,
	so $\lm - \mu$ is either $0$ or invertible.
	It follows that if $\lm\neq\mu$ then $\la x,y\ra = 0$.
\end{proof}

\begin{corollary}
	\label{cor-0-behaviour}
	If $\frak F$ is a fusion rules realised in an $\frak F$-axial algebra $A$,
	then necessarily $1\in\frak F$,
	and if further $0\in\frak F$ and $A$ is Frobenius, then $1\not\in0\star0$.
\end{corollary}
\begin{proof}
	Observe that, if $z_1,z_2\in A_0$ with respect to an $\frak F$-axis $a$,
	then $\la a,z_1z_2\ra = \la az_1,z_2\ra = \la 0,z_2\ra = 0$,
	whence $z_1z_2$ is perpendicular to $a$
	and therefore $(A_0A_0)\cap A_1 = (A_0A_0)\cap\la a\ra= 0$.
\end{proof}

The following condition holds in associative algebras
and we therefore consider it a natural property for all fusion rules.
\begin{definition}
	\label{def-seress}
	A fusion rules $\frak F$ satisfies the {\em Seress condition}
	if $0\in\frak F$, $0\star1=0$ and $0\star f=f$ for all $f\in\frak F-\{1\}$.
\end{definition}
\begin{proposition}
	\label{prop-seress}
	If $\frak F$ satisfies the Seress condition,
	then an $\frak F$-axis $a$
	associates with its $0$-eigenvectors.
	That is, for $x \in A$, $z\in A_0$, we have $a(xz) = (ax)z$.
\end{proposition}
\begin{proof}
	Suppose that $x$ is a $\lm$-eigenvector for $a$.
	If $x\in A_1$ then $x = ra$ for some $r\in R$
	and so $a(xz) = a(raz) = a0 = 0 = r0 = raz = r(aa)z = (ax)z$.
	Otherwise, by the fusion rules,
	$xz\in A_{\lm\star 0} = A_\lm$ and therefore $a(xz) = \lm xz = (ax)z$ as required.
\end{proof}

This next and last lemma is a simple version of the so-called resurrection principle.
\begin{lemma}
	\label{resurrection}
	Suppose that $a$ is an $\frak F$-axis
	and let $B\subseteq A$ be an $R$-submodule of $A$
	such that $\ad(a)B=aB\subseteq B$.
	If for $x\in A$ there exist $b_\lm,b_0\in B$
	such that $y = x + b_\lm \in A_\lm$ and $z = x + b_0 \in A_0$ respectively,
	then $x = \frac{1}{\lm}a(b_\lm - b_0) - b_\lm$,
	and in particular $x\in B$.
\end{lemma}
\begin{proof}
	Since $y$ and $z$ are eigenvectors, we have $ay = \lm y$ and $az = 0$.
	Subtracting the second from the first, $a(y-z) = \lm y$.
	Since $y - z = b_\lm - b_0$, this gives us $\frac{1}{\lm}a(b_\lm-b_0) = y = x + b_\lm$,
	which yields the claim.
\end{proof}

\section{Universal algebras on $n$ generators}
\label{sec-uni-alg}
We will construct,
in the category of (commutative, nonassociative) algebras with associative bilinear forms on $n$ (marked) generators,
an initial object.
We restrict to the case of rings $R$ that contain a subfield $\kk$.

Consider the following category $\cal C_0$. Its objects are all tuples
$(R,M,\bar a)$, where $R$ is a coefficient ring, $M$ is a commutative
$R$-algebra with an associative bilinear form, and $\bar a=(a_1,\dotsc,a_k)$
is a $n$-tuple of elements generating $M$. A morphism from $(R,M,\bar a)$ to 
$(R',M',\bar a')$ is a pair $(\phi,\psi)$, where $\phi$ is a ring 
homomorphism $R\to R'$, 
and $\psi$ is an algebra homomorphism $M\to M'$ compatible with $\phi$. 
Compatibility means that, in addition to preserving sums and products, 
$\psi$ also satisfies:
    $$\psi(rm)=\phi(r)\psi(m),$$
for all $m\in M$ and $r\in R$. We further require that $\psi(a_i)={a_i}'$
for $i\leq k, a_i\in\bar a,{a_i}'\in\bar a'$ (where, as above, $\bar a = 
(a_1,\dotsc,a_k)$ and $\bar a ' = ({a_1}',\dotsc,{a_k}')$) and that the bilinear form be
preserved in the following sense:
    $$\la\psi(l),\psi(m)\ra=\phi(\la l,m\ra ),$$
for all $l,m\in M$.

We will talk about objects $M$ and morphisms $\psi$,
viewing $R$ and $\bar a$ as attributes of $M$ and,
similarly, viewing $\phi$ as an attribute of $\psi$.
We will call $R$ the {\em coefficient ring} of $M$
and $\bar a$ the {\em marked} elements of $M$.
We will call $\phi$ the {\em coefficient homomorphism} of $\psi$
and require that $\phi$ restricted to the subfield $\kk$ is the identity.

This category $\cal C_0$ possesses an initial object, which can be
constructed as follows. Let $\hat X$ be the free commutative nonassociative magma
with generators $\hat a_i$, $i\leq k$. Let $\sim$ be the equivalence
relation on $\hat X\times\hat X$ generated by all elementary equivalences
$(\hat x\hat y,\hat z)\sim(\hat x,\hat y\hat z)$ for $\hat x,\hat y,
\hat z\in\hat X$. For $\hat x,\hat y\in\hat X$, let $[\hat x,\hat y]$
denote the $\sim$-equivalence class containing $(\hat x,\hat y)$.
Furthermore, let $[\hat X\times\hat X]$ denote the set of all
$\sim$-equivalence classes.

Let $\hat R=\kk[\{\lm_{[\hat x,\hat y]}\}_{[\hat x,\hat y]\in [\hat
X\times\hat X]}]$ be the ring of polynomials with rational
coefficients and with an infinite number of indeterminates
$\lm_{[\hat x,\hat y]}$ indexed by the $\sim$-equivalence classes.

Let $\hat M=\hat R\hat X$ be the set of all formal linear
combinations $\sum_{\hat x\in \hat X}\al_{\hat x}\hat x$ of the
elements of $\hat X$ with coefficients in $\hat R$. As usual, only
finitely many coefficients in a linear combination can be nonzero.
Then $\hat M$ is a commutative $\hat R$-algebra with the obvious
addition and with multiplication defined by the operation in $\hat
X$ and the distributive law. We can also define a bilinear form on
$\hat M$ as follows:
    $$\biggl\la\sum_{\hat x\in\hat X}\al_{\hat x}\hat x,\sum_{\hat y\in\hat X}
    \bt_{\hat y}\hat y\biggr\ra=\sum_{\hat x,\hat y\in\hat X}\al_{\hat x}\bt_{\hat y}
    \lm_{[\hat x,\hat y]}.$$
This corresponds to setting the form on the basis $\hat X$ of $\hat
M$ to be
    $$\la\hat x,\hat y\ra=\lm_{[\hat x,\hat y]}.$$
The fact that the indeterminate depends on the $\sim$-equivalent
class (and not on the individual pair) implies that thus defined
bilinear form on $\hat M$ is associative. Hence, $\hat M$ is a
commutative $\hat R$-algebra with an associative bilinear form. As a
final touch, we select the obvious marked elements in $\hat M$; they
are $\hat a_i$ for $i\leq k$, the generators of $\hat X$. Clearly, this
makes $\hat M$ an object of $\cal C_0$.

\begin{proposition}
	\label{initial}
    The algebra $\hat M$ is the initial object of $\cal C_0$.
    Namely, for every object $M$ of $\cal C_0$,
    there exists a unique morphism $\psi_M$ from $\hat M$ to $M$.
\end{proposition}

\begin{proof} We start by considering $\psi_M$ on $\hat X$. Let $X$
be the submagma in $M$ generated by the marked elements $\bar a$.
Since $\hat X$ is a free commutative magma with generators $\hat a_i$,
$i\leq k$, there is a unique magma homomorphism $\hat X\to X$
that maps $\hat a_i$ to $a_i\in\bar a$ for each $i$.
Clearly, $\psi_M$ must coincide on $\hat X$ with this homomorphism,
since the marked elements of $\hat M$ must be mapped to the marked
elements of $M$.

Next we switch temporarily to the coefficient homomorphism $\phi_M$
matching $\psi_M$.
We require that $\phi_M$ must be identity on $\kk$ and it
should send $\lm_{[\hat x,\hat y]}=\la\hat x,\hat y\ra\in\hat R$ to
$\la\psi_M(\hat x),\psi_M(\hat y)\ra\in R$ for all $\hat x,\hat y\in \hat X$. 
Here $R$ is the coefficient ring of $M$. (Note also that the value of
$\la\psi_M(\hat x),\psi_M(\hat y)\ra$ does not depend on the choice of the 
pair $(\hat x,\hat y)$ within its $\sim$-equivalence class, because the 
bilinear form on $M$ is associative.) Since $\hat R$ is the ring of
polynomials, there exists a unique homomorphism from $\hat R$ to $R$
satisfying the above conditions, which is clearly $\phi_M$.

Finally, we must extend $\psi_M$ from $\hat X$ to the entire $\hat
M$ by linearity, as follows:
$$\psi_M\biggl(\sum_{x\in \hat X}\al_xx\biggr)=\sum_{x\in \hat
X}\phi_M(\al_x)\psi_M(x).$$
Thus, the morphism $\psi_M$ exists and is unique.\end{proof}

We call an object $M$ of $\cal C_0$ {\em rigid} if both $\psi_M$ and
$\phi_M$ are surjective, that is, if the $a_i\in\bar a$ generate the
algebra $M$ and the values of the associative bilinear form on the
submagma $X$ of $M$, generated by the marked elements, generate the
coefficient ring $R$ over the base field $\kk$. We note that the
initial object $\hat M$ is rigid and, in fact,
the above argument can be generalized to show that if $M$ is rigid
then there is at most one morphism from $M$ to any other $M'$.

We only consider rigid objects of $\cal C_0$. Let $\cal C$ be a subcategory
whose objects are the class representatives of isomorphism classes of rigid objects of $\cal C_0$.
We call $\cal C$ the {\em category of $n$-generated commutative algebras with associative
forms}, and, correspondingly, we call $\hat M$ the {\em universal
$n$-generated commutative algebra with associative form}. Taking the above comments
into account, $\cal C$ can be viewed as the category of all quotients
of $\hat M$. Indeed, every algebra $M$ in $\cal C$ is the quotient, up to isomorphism, of
$\hat M$ over the ideal $I_M=\ker\psi_M$ and the coefficient ring
$R$ is the quotient of $\hat R$ over the ideal $J_M=\ker\phi_M$. For
two objects, $M$ and $M'$, of $\cal C$ a morphism from $M$ to $M'$
exists if and only if $I_M\supseteq I_{M'}$ and $J_M\supseteq
J_{M'}$.

The presence of the associative bilinear form means that there is a
certain dependence between the ideals $I_M$ and $J_M$. Given
$I\unlhd\hat M$ and $J\unlhd\hat R$, we say that these ideals {\em
match} if $I=I_M$ and $J=J_M$ for some algebra $M$ from $\cal C$.

\begin{proposition} \label{condition}
Suppose $I\unlhd\hat M$ and $J\unlhd\hat R$. Then $I$ and $J$ match
if and only if $J\ne\hat R$ and, additionally, $J\hat M\subseteq I$
and $\la i,m\ra \in J$ for all $i\in I$ and $m\in\hat M$.
\end{proposition}

\begin{proof} Suppose $I$ and $J$ match, that is, $I=I_M$ and $J=J_M$ for some
$M$ in $\cal C$. Since the coefficient ring $R$ of $M$ is an extension
of $\kk$ and $R\cong\hat R/J_M=\hat R/J$, we must have that
$J\cap\kk=\{0\}$, which implies that $J\ne\hat R$. Also, if $j\in J$
and $m\in\hat M$ then $\phi_M(j)=0$ and so
$\psi_M(jm)=0\psi_M(m)=0$, and hence $jm\in I_M=I$. Similarly, if
$i\in I$ and $m\in\hat M$ then
$\phi_M(\la i,m \ra )=\la\psi_M(i),\psi_M(m)\ra=\la0,\psi_M(m)\ra=0$, which means
that $\la i,m \ra \in J$, so the conditions above are necessary.

Conversely, suppose the conditions in the statement of the
proposition are satisfied. Set $M=\hat M/I$ and $R=\hat R/J$. Also
set $\psi_M$ and $\phi_M$ to be the corresponding projections $\hat
M\to M$ and $\hat R\to R$. Define the action of $R$ on $M$ via
$(r+J)(m+I)=rm+I$. This is well-defined, since $J\hat M\subseteq I$.
Also, define the bilinear form on $M$ by $\la m+I,n+I \ra =\la m,n \ra + I$.
Again, this is well-defined, because $\la I,M \ra \subseteq J$.
It is a routine exercise now to check that thus defined $M$
is an object of $\cal C$ and that $I=I_M$ and $J=J_M$,
yielding that $I$ and $J$ match.\end{proof}

\begin{corollary} \label{intersection}
Suppose we have families of matching ideals $I_s\unlhd\hat M$ and
$J_s\unlhd\hat R$, where $s$ runs through some index set $S$. Then
$I=\cap_{s\in S}I_s$ and $J=\cap_{s\in S}J_s$ match, too.
\end{corollary}

\begin{proof} Indeed, $J\hat M\subseteq J_s\hat M\subseteq I_s$ for each $s$.
Hence $J\hat M\subseteq I$. Similarly, for $i\in I$ and $m\in\hat
M$, we have that $i\in I_s$ and so $\la i,m \ra \in J_s$, for each $s$.
Hence $\la i,m \ra \in J$.\end{proof}

We conclude this section with a discussion of ``presentations'' of
algebras from $\cal C$. Suppose $U$ is a subset of $\hat M$ and $V$ is
a subset of $\hat R$. We are interested in the subcategory $\cal C_{U,V}$ of
all algebras $M$ from $\cal C$, for which we have $U\subseteq I_M$ and
$V\subseteq J_M$. It follows from Corollary \ref{intersection} that
if $\cal C_{U,V}$ is nonempty, that is, there exist at least one
matching pair $I$ and $J$ containing $U$ and $V$, respectively, then
there exists a unique minimal matching pair containing them. This
means that, when $\cal C_{U,V}$ is nonempty, it contains one largest
algebra $M(U,V)$, and $C_{U,V}$ is simply the category of all
quotients of $M(U,V)$. We record this in the following proposition,
where we also give a more concrete description of the minimal
matching ideals $I(U,V)$ and $J(U,V)$.

Let $I_0$ be the ideal of $\hat M$ generated by $U$. Define $J(U,V)$
to be the ideal of $\hat R$ generated by $V$ together with all
elements $\la i,m \ra $, for $i\in I_0$ and $m\in\hat M$. Set $I(U,V)$ to
be the ideal of $\hat M$ additively generated by $U$ and $J(U,V)\hat M$.

\begin{proposition}
    \label{presentation}
    The subcategory $\cal C_{U,V}$ is nonempty if and only if
    $I(U,V)\ne\hat R$. Furthermore, if the latter holds then $I(U,V)$
    and $J(U,V)$ match and hence correspond to some algebra $M(U,V)$ of
    $\cal C$. The algebra $M(U,V)$ is the initial object of $\cal C_{U,V}$ and
    $\cal C_{U,V}$ consists of all quotients of $M(U,V)$.
\end{proposition}
\begin{proof}
    Let $I=I(U,V)$ and $J=J(U,V)$. If $\cal C_{U,V}$ is nonempty, say
it contains an algebra $M$, then $U\subseteq I_M$ and $V\subseteq
J_M$. In view of Proposition \ref{condition}, $J_M$ contains all
values $\la i,m \ra $ for $i\in I_M$ and $m\in\hat M$. Since $U\subseteq
I_M$, we see that $I_0\subseteq I_M$ and so $J_M$ contains all
$\la i,m \ra $, where $i\in I_0$ and $m\in\hat M$. Since also $V\subseteq
J_M$, we conclude that $J\subseteq J_M$. Quite similarly, we argue
that $I\subseteq I_M$. It follows that $M$ is in $\cal C_{U,V}$ if and
only if $I\subseteq I_M$ and $J\subseteq J_M$.

In particular, if $J=\hat R$ then $\cal C_{U,V}$ is empty. Suppose now
that $J\ne\hat R$. To see that $I$ and $J$ match, we first notice
that by definition $I$ contains $J\hat M$. For the second part of
the condition, select $i\in I$ and $m\in\hat M$. Then $i$ can be
written as
$$i=i_0+\sum_sj_sm_s,$$
where $s$ is an index, $i_0\in I_0$, and $j_s\in J$ and $m_s\in\hat
M$ for all $s$. Hence
$$\la i,m \ra =\la i_0,m\ra+\sum_sj_s\la m_s,m\ra \in J.$$
Indeed, $\la i_0,m\ra$ is in $J$ by the definition of $J$, while each
$j_s\la m_s,m\ra$ is in $J$ because $j_s\in J$ and $J$ is an ideal. Thus,
the second part of the condition also holds and so $I$ and $J$
match, yielding the second claim. The final claim is also clear,
since we established that $M$ is contained in $\cal C_{U,V}$ if and
only if $I\subseteq I_M$ and $J\subseteq J_M$.\end{proof}

We will say that $U$ and $V$ form a {\em presentation} for the
algebra $M(U,V)$. The elements of $U$ and $V$ will be called
relators. More in particular, the elements of $U$ will be called
{\em algebra relators}, while the elements of $V$ will be called
{\em coefficient relators}.

\section{Universal Frobenius $\frak F$-axial algebras on $n$ generators}
\label{sec-uni-ax}
Throughout this section, fix a fusion rules $\frak F$ over $\kk$.
Since all Frobenius $\frak F$-axial algebras are objects of $\cal C$,
they are quotients of the universal commutative algebra $\hat M$.
Our plan is as follows: we reformulate the conditions of \ref{def-axis} as
statements that certain elements of $\hat M$ are contained in $I_M$,
or that certain elements of $\hat R$ are contained in $J_M$. This
gives us sets of relators $U\subseteq\hat M$ and $V\subseteq\hat R$,
and we define the universal Frobenius $\frak F$-axial algebra as $M(U,V)$.

We start by reformulating condition (a').

\begin{lemma} \label{a}
    An algebra $M$ from $\cal C$ satisfies condition (a') for all $m\in\bar a$
    if and only if $\hat a_i\hat a_i-\hat a_i\in I_M$ and
    $\lm_{[\hat a_i,\hat a_i]}-2\op{CC}(\frak F)\in J_M$ for all $i\leq k$.
\end{lemma}
\begin{proof}
    Clearly, for $i\leq k$, $a_ia_i=a_i$ if and only if $a_ia_i-a_i=0$. Since
    $\psi_M(\hat a_i)=a_i$, we have, equivalently,
    that $\hat a_i\hat a_i-\hat a_i\in I_M$.

    Similarly, for $i\leq k$, since $\la\hat a_i,\hat a_i\ra=\lm_{[\hat a_i,\hat a_i]}$,
    the condition that $\la a_i,a_i\ra=2\op{CC}(\frak F)$
    (or equivalently, $\la a_i,a_i\ra-2\op{CC}(\frak F)=0$) reformulates as $\lm_{[\hat
    a_i,\hat a_i]}-2\op{CC}(\frak F)\in J_M$.
\end{proof}

Let $m$ be one of the marked elements of $M$, that is, $m\in\bar a$.
In order to be able to reformulate conditions (b) and (c) for $m$,
we should be able to say things about eigenvectors and eigenspaces
of $\ad(m)$. We note that $\ad(m)$ belongs to the ring of endomorphisms
of $M$ viewed as an $R$-module. Hence we can talk about nonnegative
powers of $\ad(m)$ (where multiplication is, as usual, by composition)
as particular endomorphisms of $M$. Also, we can then talk about
endomorphism $f(\ad(m))$, where $f\in\kk[t]$ is an arbitrary
polynomial with rational coefficients. For example, if $f=t^3+2t-3$
then $f(\ad(m))$ sends $n\in M$ to $m(m(mn))+2mn-3n$. (We must use
parentheses, because $M$ is nonassociative.) In any case, $f(\ad(m))n$
is again an element of $M$, so we can use this sort of formula to
express in a compact form particular elements of $M$. The following
simple observations will prove useful.

\begin{lemma} \label{polynomial}
    \begin{itemize}
        \item[{\rm (1)}] Suppose $f=gh$ for $f,g,h\in\kk[t]$.  Then
            $f(\ad(m))n=g(\ad(m))(h(\ad(m))n)=h(\ad(m))(g(\ad(m))n)$ for all $m,n\in M$.
        \item[{\rm (2)}] If $\psi$ is a morphism from $M'$ to $M$ and
            $m',n'\in M'$ then $\psi(f(\ad_{m'})n')=f(\ad_{\psi(m')})\psi(n')$.\qed
    \end{itemize}
\end{lemma}

Note that part (2) will be applied with $M'=\hat M$ and
$\psi=\psi_M$.

We will also need some standard linear algebra facts. Suppose $V$ is
a vector space over a field $\kk$ and suppose $\phi $ is a linear mapping
$V\to V$. For $\mu\in K$, let $V_\mu = V_\mu(\phi )$ be the eigenspace of $\phi $
corresponding to $\mu$. Note that $\mu$ is an eigenvalue of $\phi $ if
and only if $V_\mu\ne 0$.

\begin{lemma} \label{value}
    Suppose $\mu\in K$ and $f\in K[t]$. Then $f(\phi )$ acts on $V_\mu$ via
    multiplication with $f(\mu)$. In particular, $f(\phi )V_\mu=0$ if $\mu$
    is a root of $f$, and $f(\phi )V_\mu=V_\mu$, otherwise.\qed
\end{lemma}

We will now derive some consequences.

\begin{lemma} \label{in}
    Suppose $\mu_1,\ldots,\mu_n$ are distinct elements of $\kk$. Then
    $v\in V$ is contained in $\oplus_{i=1}^nV_{\mu_i}$ if and only
    if $f(\phi )v=0$, where $f=\prod_{i=1}^n(t-\mu_i)\in K[t]$.
\end{lemma}

\begin{proof}
    Suppose first that $v\in\oplus_{i=1}^nV_{\mu_i}$. This means
    that $v=\sum_{i=1}^nv_i$, where each $v_i\in V_{\mu_i}$. By Lemma
    \ref{value}, since every $\mu_i$ is a root of $f$, we have that
    $f(\phi )v_i=0$ for all $i$. Consequently, $f(\phi )v=0$.

    Conversely, suppose $f(\phi )v=0$. Let $f_i=f/(t-\mu_i)$, for
    $i=1,\ldots,n$. Since there is no common factor of the polynomials $f_i$,
    there exist $c_i\in K[t]$ such that
    $c_1f_1+\ldots+c_nf_n=1$. Set $v_i=c_i(\phi )(f_i(\phi )v)$ for all $i$.
    Clearly, $v=\sum_{i=1}^nv_i$. Also, setting $g_i=t-\mu_i$, we see
    that $f_ig_i=f$ and hence
    $\phi v_i-\mu_iv_i=g_i(\phi )v_i=g_i(\phi )(c_i(\phi )(f_i(\phi )v))=c_i(\phi )(f(\phi )v)=c_i(\phi )0=0$.
    Thus, $\phi v_i-\mu_iv_i=0$, or equivalently, $\phi v_i=\mu_iv_i$. Hence
    $v_i\in V_{\mu_i}$ for all $i$. This shows that
    $v\in\oplus_{i=1}^nV_{\mu_i}$.
\end{proof}

\begin{corollary} \label{complete}
    If the $\mu_i$ and $f$ are as above, then $V=\oplus_{i=1}^nV_{\mu_i}$
    if and only if $f(\phi )=0$.\qed
\end{corollary}

We will also need the following consequence of Lemma \ref{value}.

\begin{corollary} \label{image}
	Suppose $V=\oplus_{i=1}^nV_{\mu_i}$ and $g\in K[t]$. Then
	$g(\phi )V=\oplus_{i\in I_g}V_{\mu_i}$, where
	$I_g=\{i\in I=\{1,\ldots,n\}\mid g(\mu_i)\ne 0\}$.\qed
\end{corollary}

We are now prepared to express properties (b) and (c) in terms
of relators.

\begin{lemma} \label{b}
    Let $f=\prod_{\theta\in\frak F-\{1\}}(t-\theta)$. An algebra $M$ from $\cal C$,
    that is known to satisfy condition (a) for $a$ and $b$, also
    satisfies condition (b) for these elements if and only if
    $f(\ad_{\hat m})(\hat x-\la \hat m,\hat x \ra \hat m)\in I_M$
    for all $\hat m=\hat a_i$, $i\leq k$, and all $\hat x\in\hat X$.
\end{lemma}
\begin{proof}
	Suppose $M$ satisfies (b), that is, for $m\in\bar a$,
	every $n\in m^\perp$ is a sum of eigenvectors of $\ad(m)$
	with eigenvalues in $\frak F-\{1\}$.
	This means that $m^\perp\subseteq M_{\frak F-\{1\}}$.
	Lemma \ref{in} yields that $f(\ad(m))n=0$ for all
	$n\in m^\perp$.

	If $n$ is an arbitrary element of $M$ then $n-\la m,n \ra m\in m^\perp$.
	Hence $f(\ad(m))(n-\la m,n \ra m)=0$ for all $n\in M$. In particular,
	$f(\ad(m))(x-\la m,x \ra m)=0$ for all $m=a_i\in\bar a$ and all $x\in X$,
	where $X$ is, as before, the submagma in $M$ generated by the $a_i\in\bar a$.
	In terms of $\hat M$ this translates as the statement that
	$f(\ad_{\hat m})(\hat x-\la\hat m,\hat x \ra\hat m)\in I_M$
	for all $\hat m=\hat a_i$, $i\leq k$, and all $\hat x\in\hat X$.

	Conversely, suppose $f(\ad_{\hat m})(\hat x-\la\hat m,\hat x \ra\hat m)\in I_M$
	for all $\hat m=\hat a_i$,$i\leq k$, and all $\hat x\in\hat X$.
	This translates back to give us that $f(\ad(m))(x-\la m,x \ra m)=0$
	for all $m\in\bar a$ and all $x\in X$.
	Note that $X$ spans $M$ (since $\hat X$ spans $\hat M$) and, therefore,
	the elements $x-\la m,x \ra m$, for $x\in X$, span $m^\perp$.
	It follows that $f(\ad(m))n=0$ for all $n\in m^\perp$.
	According to Lemma \ref{in}, this means that
	$m^\perp\subseteq M_{\frak F-\{1\}}$,
	that is, every $n\in m^\perp$ is a sum of eigenvectors as in (b).
\end{proof}

Recall that the condition
$m^\perp\subseteq M_{\frak F-\{1\}}$
is equivalent to
$m^\perp=M_{\frak F-\{1\}}$.
This is because $m$ is contained in $M_{1}$, and so the
latter complements $m^\perp$.

We now turn to condition (c),
the restriction imposed by the fusion rules $\star$.
The following lemma takes care of the value of $\star$ on two fields.

Given an algebra $M$ from $\cal C$, pick $m\in\bar a$ and let
$\hat m$ be the corresponding marked element of $\hat M$, that
is, $\psi_M(\hat m)=m$. Let
$\emptyset\ne\{\theta_1,\ldots,\theta_n\}\subseteq\frak F$.
Let, as above,
$f=\prod_{f\in\frak F-\{1\}}(t-f)$ and write $f_\mu=f/(t-\mu)$.
Finally, set $g=\prod_{i=1}^n(t-\theta_i)$.
Let $\mu,\nu\in\frak F-\{1\}$ identify a case of
fusion rules, that is, an entry in the table.
(Note that $f_\mu,f_\nu,g\in\kk[t]$.)

\begin{lemma} \label{cell}
    Suppose $M$ satisfies (a') and (b) for all $a_i\in\bar a$ and suppose
    that $m$, $\hat m$, $\mu$, $\nu$, the $\theta_i$, $f_\mu$, $f_\nu$,
    and $g$ are as above. Then
        $$M_{\mu}(m) M_{\nu}(m) \subseteq\oplus_{i=1}^nM_{\theta_i}(m)$$
    if and only if
        $$g(\ad_{\hat m})((f_\mu(\ad_{\hat m})(x-(\hat m,x)\hat m))(f_\nu(\ad_{\hat m})(x'-(\hat m,x')\hat m)))\in I_M,$$
    for all $x,x'\in\hat X$.
\end{lemma}
\begin{proof}
    The latter formula looks complicated, but it is in fact easy to decipher.
    First of all,
        $$g(\ad_{\hat m})((f_\mu(\ad_{\hat m})(x-(\hat m,x)\hat m))(f_\nu(\ad_{\hat m})(x'-(\hat m,x')\hat m)))\in I_M,$$
    for all $x,x'\in\hat X$, if and only if
        $$g(\ad(m))((f_\mu(\ad(m))(x-\la m,x \ra m))(f_\nu(\ad(m))(x'-\la m,x' \ra m)))=0,$$
    for all $x,x'\in X$, where, as before, $X$ is the submagma in $M$ generated by
    the $a_i\in\bar a$.

    Next we apply Lemma \ref{in}, which yields that, equivalently,
        $$(f_\mu(\ad(m))(x-\la m,x \ra m))(f_\nu(\ad(m))(x'-\la m,x' \ra m))\in\oplus_{i=1}^nM_{\theta_i}(m),$$
    for all $x,x'\in X$. Finally, we have already noticed that the elements
    $x-\la m,x \ra m$, where $x\in X$, span $m^\perp$. Hence, by Corollary \ref{image},
    the elements $f_\mu(\ad(m))(x-\la m,x \ra m)$, where $x\in X$, span $M_{\mu}$.
    (Here we use that $M$ satisfies condition (b) for $m$, and hence
    $m^\perp=M_{\frak F-\{1\}}$.)
    Similarly, the elements $f_\nu(\ad(m))(x'-\la m,x' \ra m)$, where $x'\in X$, span
    $M_{\nu}$. Hence the products
    $(f_\mu(\ad(m))(x-\la m,x \ra m))(f_\nu(\ad(m))(x'-\la m,x' \ra m))$ span $M_{\mu}M_{\nu}$.
    Thus, all the conditions above are equivalent simply to
        $$M_{\mu}M_{\nu} \subseteq\oplus_{i=1}^nM_{\theta_i},$$
    and this is exactly what the lemma claims.
\end{proof}

Set $f_{\lm,\mu} = \prod_{\nu\in\lm\star\mu}(t-\nu)$.
\begin{corollary} \label{c}
    Suppose that $M$ satisfies (a) and (b) for some $m\in\bar a$. Then it
    also satisfies (c) for $m$ if and only if 
    for all $\lm,\mu\in\frak F-\{1\}$, and $x,x'\in\hat X$,
    \begin{equation}
	   f_{\lm,\mu}(\ad(m))\bigl(
		(f_\lm(\ad(m))(x-\la m,x \ra m))
		(f_\mu(\ad(m))(x'-\la m,x' \ra m))
	   \bigr) \in I_M.
    \end{equation}
    Hence $M$ is a $\frak F$-axial algebra exactly when
    this set of equations is true for every $m\in\bar a$. \qed
\end{corollary}

Combining Proposition \ref{presentation} with Lemmas \ref{a} and \ref{b} and
Corollary \ref{c} we finally arrive at the main result of this section.

\begin{theorem} \label{universal axial}
	The category $\cal M$ of Frobenius $\frak F$-axial algebras
	coincides with $\cal C_{U,V}$, where
	$V$ consists of the coefficient relators from Lemma {\rm\ref{a}} and $U$
	consists of all algebra relators from Lemmas {\rm\ref{a}}, {\rm\ref{b}}
	and Corollary {\rm\ref{c}}. In particular, $\tilde M=M(U,V)$ is the universal
	$\frak F$-axial algebra, of which all other $\frak F$-axial algebras are quotients.\qed
\end{theorem}

The universal $n$-generated Frobenius $\frak F$-axial algebra $\tilde M$
has additional symmetries.
For this we need a broader concept of automorphisms.
Given an automorphism $\phi$ of $R$ fixing the subfield $\kk$,
we define a {\em $\phi$-automorphism} $\psi$ of $M$
to be a bijective mapping $M\to M$ preserving addition and multiplication
and such that
$$(rm)^\psi=r^\phi m^\psi \text{ and }
\la m^\psi,n^\psi\ra=\la m,n \ra ^\phi,$$
for all $r\in R$ and $m,n\in M$. We will call $\phi$-automorphisms,
for various $\phi$, {\em semiautomorphisms} of $M$. Note that every
automorphism of $M$ is a semiautomorphism for $\phi=1$ and that
all semiautomorphisms form a group $\op{SAut} M$; it contains $\Aut M$ as a
normal subgroup.

The universal algebra $\tilde M$ possesses a special group of semiautomorphisms
arising from permutations of the marked elements $a_i\in\bar a$.
Take $\sg\in\Sym_k$.
We will first construct the action of $\sg$ on $\hat X$, $\hat R$ and $\hat M$
and then note that this action descends to  $\tilde M$
and the corresponding ring $\tilde R$.

The mapping $\hat a_i\mapsto\hat a_{i^\sg}$ clearly extends to an automorphism 
$\hat\psi_\sg$ of the free commutative magma $\hat X$. It also quite clearly 
preserves the equivalent relation $\sim$ on $\hat X\times\hat X$ and, hence, 
leads to a permutation of equivalence classes $[\hat x,\hat y]$. Since the 
equivalence classes parametrize the indeterminates $\lm_{[\hat x,\hat y]}$ 
of the polynomial ring $\hat R$, the permutation of the classes induces an 
automorphism $\hat\phi_\sg$ of $\hat R$. Since $\hat X$ is a basis of $\hat M$, 
there is only one way to extend $\hat\psi_\sg$ to the entire $\hat M$, namely,
$$\biggl(\sum_{\hat x\in\hat X}a_{\hat x}\hat x\biggr)^{\hat\psi_\sg}=
\sum_{\hat x\in\hat X}a_{\hat x}^{\hat\phi_\sg}\hat x^{\hat\psi_\sg}.$$
Manifestly, the resulting mapping $\hat\psi_\sg$
is a $\hat\phi_\sg$-automorphism as defined above.
It is also easy to see that the sets $U$ and $V$
from Proposition \ref{universal axial} are invariant
under $\hat\phi_\sg$ and $\hat\psi_\sg$, respectively.
Hence the matching ideals $I(U,V)$ and $J(U,V)$ are also invariant
and so $\hat\phi_\sg$ and $\hat\psi_\sg$ descend to automorphisms
$\phi_\sg$ and $\psi_\sg$ of $\tilde R$ and $\tilde M$, respectively. 
The map $\sg\mapsto\psi_\sg$ is an action of $\Sym_k$ on $\tilde M$. We will 
refer to the semiautomorphisms $\psi_\sg$ as the permutations of the 
marked elements.

The universal algebra $\tilde M$ is a natural object of interest.
In general, we would like a description of $\tilde M$,
starting with whether $\tilde M$ is finite-dimensional.
In all known cases it is.
There is one case where we have a complete description of $\tilde M$,
and it occupies our attention in the sequel.

\section{The $2$-generated Frobenius $\frak V(4,3)$-case}
\label{sec-two-gen}

For the rest of this paper, we consider the $2$-generated case
with the fusion rules $\frak V(4,3)$ over $\kk = \QQ$.
Hence $k=2$ and $\tilde M$ is the $2$-generated universal Frobenius $\frak V(4,3)$-axial algebra.
Recall also that, by Corollary \ref{cor-0-behaviour},
we have that $0\star0 = 0$,
which is a refinement of the rules given previously.

Let $M$ be a Frobenius $\frak V(4,3)$-axial algebra
with two marked generators $a_0$ and $a_1$.
For each of $a_0$, $a_1$, there is a Miyamoto involution
(as in Lemma \ref{prop-miyamoto-inv}),
written $\tau_0=\tau(a_0)$ and $\tau_1=\tau(a_1)$.
Inside $\Aut M$, $\tau_0$, $\tau_1$ generate a (finite or infinite) dihedral subgroup $T_0$.
Let $\rho=\tau_0\tau_1$.

For $i\in\ZZ$, set $a_{2i}={a_0}^{\rho^i}$ and $a_{2i+1}={a_1}^{\rho^i}$.
Every semiautomorphism takes a $\frak V(4,3)$-axis to a $\frak V(4,3)$-axis.
Hence all $a_j$ are $\frak V(4,3)$-axes.
We set $\tau_j=\tau(a_j)$.
Observe that $\tau_{2i}={\tau_0}^{\rho^i}$ and $\tau_{2i+1}={\tau_1}^{\rho^i}$,
where the exponential notation stands for (right) conjugation in $\Aut M$.
More generally, we have the following.

\begin{lemma} \label{action}
	For $i\in\ZZ$, we have ${a_i}^\rho=a_{i+2}$ and
	${\tau_i}^\rho=\tau_{i+2}$. Also, ${a_i}^{\tau_0}=a_{-i}$,
	${a_i}^{\tau_1}=a_{2-i}$ and, correspondingly,
	${\tau_i}^{\tau_0}=\tau_{-i}$, ${\tau_i}^{\tau_1}=\tau_{2-i}$.\qed
\end{lemma}

There is only one semiautomorphism
coming from the action of $\Sym_2$ on the marked generators:
this involution will be referred to as the {\em flip}.
We will use the bar notation for the flip
for both its action on the ring and on the algebra.

The flip semiautomorphism of $M$ swaps $a_0$ and $a_1$ and hence
interchanges $\tau_0$ and $\tau_1$. This means that the flip normalizes
the dihedral group $T_0$ and thus induces on it an automorphism.
In particular, we can form a possibly larger dihedral
group $T$ generated by $\tau_0$ and the flip.
Then $T_0$ is a subgroup of index at most two in $T$.

The groups $T_0$ and $T$ have their counterparts at the level of
$\hat M$. Let $\hat T$ be the infinite dihedral group that acts on
$\ZZ$ and is generated by the reflections $\hat\tau_0\colon i\mapsto -1$ and
$\hat f\colon i\mapsto 1-i$. Also consider $\hat\tau_1\colon i\mapsto 2-i$.
Clearly, the mapping from $\hat T$ to $T$ sending $\hat\tau_0$ to
$\tau_0$ and $\hat f$ to the flip is a homomorphism of groups.
Furthermore, this homomorphism takes $\hat\tau_i$ to $\tau_i$.
Accordingly, we set $\hat T_0=\la\hat\tau_0,\hat\tau_1\ra$ and we have
that the above homomorphism maps $\hat T_0$ onto $T_0$.

When $k$ is odd, $\hat T_0$ acts transitively on all unordered pairs
$\{i,j\}$ for which $|i-j|=k$. When $k$ is even, $\hat T_0$ has two orbits
on the set of pairs $\{i,j\}$ for which $|i-j|=k$. One orbit consists of
those pairs where $i$ and $j$ are even, and the other consists of
the pairs where $i$ and $j$ are odd. Note that $\hat f$ interchanges
these two orbits.
Clearly, if $\{i,j\}$ and $\{i',j'\}$ are in the same $\hat T_0$-orbit
then $\{a_i,a_j\}$ and $\{a_{i'},a_{j'}\}$ are in the same $T_0$-orbit.
This leads to the following result.

\begin{lemma} \label{lambdas}
	The value of $\la a_i,a_j\ra$ depends only on $k=|i-j|$, if $k$ is odd.
	If $k$ is even then $\la a_i,a_j\ra$ depends only on $k$ and, possibly,
	on the parity of $i$ (which is the same as the parity of $j$).
\end{lemma}
\begin{proof}
	The value of the invariant form is preserved by the action of $T_0$,
	so the claim follows from the remarks before the lemma.
\end{proof}

Adding to the above the action of the flip allows us to be a little
more specific.

\begin{lemma}
	\label{flipped lambdas}
	If $i-j$ is odd then $\overline{\la a_i,a_j\ra}=\la a_i,a_j\ra$.
	If $i-j$ is even then $\overline{\la a_i,a_j\ra}=\la a_{i+1},a_{j+1}\ra$.
\end{lemma}
\begin{proof}
	This follows since $\hat f$ (which maps to the flip) preserves the
	only orbit for $k=i-j$ odd and it interchanges the two orbits for
	$k$ even.
\end{proof}

It follows from these lemmas that, when $k$ is odd, there exists
$\nu_k\in R$, such that $\la a_i,a_{i+k}\ra=\nu_k$ for all $i\in\ZZ$.
Furthermore, $\bar\nu_k=\nu_k$. Similarly, for $k$ even, there exist
$\nu_k^e,\nu_k^o\in R$, such that $\la a_i,a_{i+k}\ra=\nu_k^e$ or
$\nu_k^o$ depending on whether $i$ is even or odd, and
$\bar\nu_k^e=\nu_k^o$. By $\frak F$-axis condition (a'), $\nu_0^e=\nu_0^o=1$.

We will now define some special elements of $M$ that are invariant
under the action of $T_0$.
We will need the following observation,
which was used to great effect in \cite{sakuma}.

\begin{lemma} \label{sigma}
	If $m$ and $n$ are two $\frak V(4,3)$-axes of $M$ then
	$\sg(m,n)=mn-\frac{1}{32}(m+n)$ is invariant under both $\tau(m)$ and
	$\tau(n)$.
\end{lemma}

\begin{proof}
	By symmetry it suffices to show that $\sg(m,n)$ is invariant
	under $\tau=\tau(m)$.

	Recall that $n-\la n,m \ra m\in m^\perp$ and so
	$n-\la m,n \ra m=\al(n)+\bt(n)+\gm(n)$, where the three summands belong
	to $M_{0}$, $M_{\frac{1}{4}}$, and $M_{\frac{1}{32}}$,
	respectively. Recall also that $\tau$ acts as identity on
	$M_+=M_{1}\oplus M_{0}\oplus M_{\frac{1}{4}}$. Since
	$n=\la m,n \ra m+\al(n)+\bt(n)+\gm(n)$, we have that
	$mn=\ad_mn=\la m,n \ra m+\frac{1}{4}\bt(n)+\frac{1}{32}\gm(n)$. Hence
	$\sg(m,n)=nm-\frac{1}{32}(m+n)=
	\frac{1}{32}(31\la m,n \ra -1)m-\frac{1}{32}\al(n)+\frac{7}{32}\bt(n)$ is
	fully contained in $M_+$, and so it is fixed by $\tau$.
\end{proof}

Define $\sg_1=\sg(a_0,a_1)$.

\begin{lemma} \label{sigma one}
	We have that $\sg_1$ is fixed by $T$. In particular,
	$\sg(a_i,a_{i+1})=\sg_1$ for all $i\in\ZZ$.
\end{lemma}
\begin{proof}
	By Lemma \ref{sigma}, $\sg_1$ is invariant under $\tau_0$ (and
	$\tau_1$). It is also invariant under the flip,
	which swaps $a_0$ and $a_1$, and fixes the rationals.
	Therefore $\sg_1$ is invariant under $T$.
	The second claim follows by the transitivity of $\hat T$.
\end{proof}

Equivalently, we can write that
$a_ia_{i+1}=\sg_1+\frac{1}{32}(a_i+a_{i+1})$ for all $i$. Two
further invariant elements come from distance two. Namely, define
$\sg_2^e=\sg(a_0,a_2)$ and $\sg_2^o=\sg(a_{-1},a_1)$.

\begin{lemma} \label{sigma two}
	The elements $\sg_2^e$ and $\sg_2^o$ are invariant under the action
	of $T_0$ and they are interchanged by the flip. Moreover,
	$\sg(a_i,a_{i+2})=\sg_2^e$ or $\sg_2^o$ depending on whether $i$ is even
	or odd.
\end{lemma}
\begin{proof}
	By Lemma \ref{sigma}, $\sg_2^e$ is invariant under $\tau_0$.
	Also,
	$(\sg_2^e)^{\tau_1}=\sg(a_0,a_2)^{\tau_1}=\sg(a_0^{\tau_1},a_2^{\tau_1})=\sg(a_2,a_0)=\sg(a_0,a_2)=\sg_2^e$.
	So $\sg_2^e$ is also invariant under $\tau_1$, and hence it is fixed
	by the entire $T_0$. Similarly, $T_0$ fixes $\sg_2^o$. Clearly, the
	flip interchanges $\sg_2^e$ and $\sg_2^o$. The final claim follows as
	before.
\end{proof}

\section{The multiplication table}
\label{sec-cal}

The first five lemmas here reprove,
for our situation, computations from \cite{ipss}.
These determine products in the specialised $\frak V(4,3)$-case
of the $\tilde M$ construction of Section \ref{sec-uni-ax};
here, we will simply call the universal object $U$.
The proof of Theorem \ref{thm-sakuma} is grounded in the computations
that we undertake in this section.

Since $\nu_1$ frequently appears in this text, we give it a shorter name:
let $\lm=\nu_1$. Note that $\bar\lm=\lm$.

\begin{lemma}
	\label{formulas}
	We write $a_1=\lm a_0+\al_1+\bt_1+\gm_1$ for the decomposition of $a_1$
	into $1$,$0$,$\frac{1}{4}$ and $\frac{1}{32}$-eigenvectors respectively. The following hold:
	\begin{align}
		\label{eq-al}
		\al_1&=-4\sg_1+(3\lm-\frac{1}{8})a_0+\frac{7}{16}(a_1+a_{-1}), \\
		\label{eq-bt}
		\bt_1&=4\sg_1-(4\lm-\frac{1}{8})a_0+\frac{1}{16}(a_1+a_{-1}), \\
		\label{eq-gm}
		\gm_1&=\frac{1}{2}(a_1-a_{-1}).
	\end{align}
\end{lemma}
\begin{proof}
	Applying $\tau_0$ to both sides of the equality
	$a_1=\lm a_0+\al_1+\bt_1+\gm_1$, we get
	$a_{-1}=\lm a_0+\al_1+\bt_1-\gm_1$.
	(Hence, using analogous notation,
	$\al_{-1}=\al_1$, $\bt_{-1}=\bt_1$ and $\gm_{-1}=-\gm_1$.)
	Hence $\gm_1=\frac{1}{2}(a_1-a_{-1})$, proving \eqref{eq-gm}.

	So $a_0 a_1=\lm a_0+\frac{1}{4}\bt_1+\frac{1}{32}\gm_1=
	\lm a_0+\frac{1}{4}\bt_1+\frac{1}{64}(a_1-a_{-1})$. Solving this for
	$\frac{1}{4}\bt_1$ and using $a_0 a_1=\sg_1+\frac{1}{32}(a_0+a_1)$,
	we get $\frac{1}{4}\bt_1=
	\sg_1-\lm a_0-\frac{1}{64}(a_1-a_{-1})+\frac{1}{32}(a_0+a_1)=
	\sg_1-(\lm -\frac{1}{32})a_0+\frac{1}{64}(a_1+a_{-1})$, yielding \eqref{eq-bt}.
	Result \eqref{eq-al} now follows by substitution of \eqref{eq-bt} and \eqref{eq-gm}
	into $a_1=\lm a_0+\al_1+\bt_1+\gm_1$.
\end{proof}

\begin{lemma}
	\label{times sigma}
	We have $a_0\sg_1=\frac{7}{32}\sg_1+\left(\frac{3}{4}\lm-\frac{25}{2^{10}}\right)a_0+
	\frac{7}{2^{11}}(a_{-1}+a_1)$.
\end{lemma}
\begin{proof}
	Multiplying both sides of the equality \eqref{eq-al} from Lemma \ref{formulas}
	with $a_0$, we get
	$0=-4a_0\sg_1+(3\lm-\frac{1}{8})a_0+\frac{7}{16}(a_0 a_1+a_0 a_{-1})$.
	This, after substituting $a_0 a_{\pm1}=\sg_1+\frac{1}{2^5}(a_0+a_{\pm1})$
	and collecting terms, gives $4a_0\sg_1=
	\frac{7}{8}\sg_1+\left(3\lm-\frac{25}{2^8}\right)a_0+\frac{7}{2^9}(a_{-1}+a_1)$,
	proving the claim.
\end{proof}

\begin{lemma}
	\label{with sg1}
	For all $k$, we have $\la a_k,\sg_1\ra=\frac{1}{32}(31\lm-1)$.
\end{lemma}
\begin{proof}
	Note that $\lm=\la a_0,a_1\ra=\la a_0 a_0,a_1\ra=
	\la a_0,a_0 a_1\ra=\la a_0,\sg_1\ra+\frac{1}{32}\left(\la a_0,a_0\ra+\la a_0,a_1\ra\right)$.
	Since $\la a_0,a_0\ra=1$ and $\la a_0,a_1\ra=\lm$, we conclude that
	$\lm=\la a_0,\sg_1\ra+\frac{1}{32}(1+\lm)$, which gives us
	$\la a_0,\sg_1\ra=\frac{1}{32}(31\lm-1)$.

	Since $T$ is transitive on the $a_i$ and fixes $\lm$, we have
	the same equality for all values of $k$.
\end{proof}

We can now strengthen Lemma \ref{flipped lambdas}.

\begin{lemma}
	\label{no parity}
	When $k$ is even, $\nu_k^e=\nu_k^o$. 
	Consequently, for every $k$ odd or even,
	there is $\nu_k\in R$, satisfying $\bar\nu_k=\nu_k$,
	such that $\la a_i,a_{i+k}\ra=\nu_k$ for all $i$.
\end{lemma}
\begin{proof}
	We only need to prove the first claim. Suppose that $k$ is even. If
	$k=0$ then $\nu_0^e=1=\nu_0^o$. Hence we can assume that $k$ is at least $2$.
	By applying the flip to Lemma \ref{times sigma}, we find that
	\begin{equation}
		a_1\sg_1=\frac{7}{32}\sg_1+\left(\frac{3}{4}\lm-\frac{25}{2^{10}}\right)a_1+\frac{7}{2^{11}}(a_2+a_0).
	\end{equation}
	Applying the form with $a_{k-1}$ to each side of the equality in \ref{times sigma} for $a_0\sg_1$,
	and similarly applying $\la a_k,\_\ra$ to the expression for $a_1\sg_1$,
	we obtain two further equations:
	\begin{align}
	\la a_{k-1},a_0\sg_1\ra&=\frac{7}{2^5}\la a_{k-1},\sg_1\ra+\left(\frac{3}{2^2}\lm-\frac{25}{2^{10}}\right)\nu_{k-1}+\frac{7}{2^{11}}(\nu_k^o+\nu_{k-2}^o), \\
	\la a_k,a_1\sg_1\ra&=\frac{7}{2^5}\la a_k,\sg_1\ra+\left(\frac{3}{2^2}\lm-\frac{25}{2^{10}}\right)\nu_{k-1}+\frac{7}{2^{11}}(\nu_k^e+\nu_{k-2}^e).
	\end{align}

	Since $k-1$ is odd, as previously remarked $T_0$ contains an element
	$\tau$ mapping $\{a_0,a_{k-1}\}$ to $\{a_1,a_k\}$. 
	Because the algebra product is commutative,
	we have $(a_0 a_{k-1})^\tau=a_1 a_k$.

	As the form and $\sg_1$ are $T_0$-invariant,
	\begin{equation}
		\la a_{k-1},a_0\sg_1\ra
	 	=\la a_{k-1} a_0,\sg_1\ra
	 	=\la (a_{k-1} a_0)^\tau,{\sg_1}^\tau\ra=\la a_k a_1,\sg_1\ra
		=\la a_k,a_1 \sg_1\ra.
	\end{equation}
	Also, by Lemma \ref{with sg1}, $\la a_{k-1},\sg_1\ra=\la a_k,\sg_1\ra$.
	Therefore, from the two equalities above, we get
	$\frac{7}{2^{11}}(\nu_k^o+\nu_{k-2}^o)=\frac{7}{2^{11}}(\nu_k^e+\nu_{k-2}^e)$,
	that is, $\nu_k^o+\nu_{k-2}^o=\nu_k^e+\nu_{k-2}^e$. This means that
	$\nu_k^o-\nu_k^e=-\nu_{k-2}^o+\nu_{k-2}^e=-(\nu_{k-2}^o-\nu_{k-2}^e)$.
	Since this holds for all even $k\ge 2$ and since $\nu_0^o-\nu_0^e=0$,
	we conclude by induction that $\nu_k^o-\nu_k^e=0$ for all even $k$.
\end{proof}

The value $\nu_2$ will also frequently appear,
so we give it the special name $\mu=\nu_2$.

\begin{lemma}
	\label{sg1 with sg1}
	We have $\la\sg_1,\sg_1\ra=\frac{3}{4}\lm^2+\frac{65}{2^9}\lm-\frac{3}{2^{11}}+\frac{7}{2^{11}}\mu$.
\end{lemma}
\begin{proof}
	Starting from the equality in Lemma \ref{times sigma}, we get
	\begin{equation}
		\begin{aligned}
			\la a_1,a_0\sg_1\ra & = 
			\frac{7}{32}\la a_1,\sg_1\ra+
			\left(\frac{3}{4}\lm-\frac{25}{2^{10}}\right)\la a_1,a_0\ra+
			\frac{7}{2^{11}}(\la a_1,a_{-1}\ra+\la a_1,a_1\ra) \\
			& = \frac{7}{2^{10}}(31\lm-1)+
			\left(\frac{3}{4}\lm-\frac{25}{2^{10}}\right)\lm+
			\frac{7}{2^{11}}(\mu+1)=\frac{3}{4}\lm^2+\frac{3}{16}\lm+\frac{7}{2^{11}}(\mu-1).
		\end{aligned}
	\end{equation}

	Next, using Lemma \ref{with sg1},
	\begin{equation}
		\la a_1,a_0\sg_1\ra=\la a_1 a_0,\sg_1\ra=
		\la\sg_1,\sg_1\ra+\frac{1}{32}(\la a_0,\sg_1\ra+\la a_1,\sg_1\ra)=
		\la\sg_1,\sg_1\ra+\frac{1}{2^9}(31\lm-1).
	\end{equation}

	Thus,
	\begin{equation}
    \la\sg_1,\sg_1\ra+\frac{1}{2^9}(31\lm-1)=
		\frac{3}{4}\lm^2+\frac{3}{16}\lm+\frac{7}{2^{11}}(\mu-1),
	\end{equation}
	which after collecting terms gives us the claim.
\end{proof}

\begin{lemma}
	\label{lem a0 sg2}
	We have $a_0\sg_2^e=\frac{7}{32}\sg_2^e+\left(\frac{3}{4}\mu-\frac{25}{2^{10}}\right)a_0+\frac{7}{2^{11}}(a_{-2}+a_2)$.
\end{lemma}
\begin{proof}
	Set
	\begin{align}
		\label{eq-al2}
		\al_2 & =-4\sg_2 + (3\lm-\frac{1}{8})a_0 + \frac{7}{16}(a_2+a_{-2}), \\
		\label{eq-bt2}
		\bt_2 & = 4\sg_2 - (4\lm-\frac{1}{8})a_0 + \frac{1}{16}(a_2+a_{-2}).
	\end{align}
	These terms are the projections of $a_2$
	onto the $0$- and $\frac{1}{4}$-eigenspaces of $a_0$, respectively,
	just as we found $\al_1$ and $\bt_1$ defined in Lemma \ref{formulas}.
	Multiplying both sides of \eqref{eq-al2} with $a_0$, we get
	$0=-4a_0\sg_2^e+(3\lm-\frac{1}{8})a_0+\frac{7}{16}(a_0 a_2+a_0 a_{-2})$.
	This, after substituting for $\sg_2^e$ and collecting terms, gives
	$4a_0\sg_2= \frac{7}{8}\sg_2+\left(3\lm-\frac{25}{2^8}\right)a_0+\frac{7}{2^9}(a_{-2}+a_2)$,
	proving the claim.
\end{proof}

\begin{lemma}
	\label{second sigma}
	We have 
	\begin{multline*}
	a_0\sigma_2^o=-\frac{1}{3}\biggl[\left(-32\lm+\frac{19}{16}\right)\sigma_1-\frac{7}{32}\sigma_2^e+\\
	\left(32\lm^2-5\lm+\frac{1}{8}\mu+\frac{127}{2^{10}}\right)a_0+\left(-\frac{1}{2}\lm+
	\frac{19}{2^{10}}\right)(a_1+a_{-1})-\frac{7}{2^{11}}(a_2+a_{-2})\biggr].
	\end{multline*}
\end{lemma}
\begin{proof}
	Let $\al_1 = \al_1$ and $\bt_1 = \bt_1$
	be the projection of $a_1$ onto the $0$- and $\frac{1}{4}$-eigenspace
	respectively of $a_0$, as given in Lemma \ref{formulas}.
	By the fusion rules (c),
	$\al_1\al_1$ is a $0$-eigenvector,
	$\bt_1\bt_1$ is a sum of $0$- and $1$-eigenvectors
	and therefore $\bt_1\bt_1-\la\bt_1\bt_1,a_0\ra a_0$ is a $0$-eigenvector.
	Thus $a_0(\al_1\al_1-\bt_1\bt_1+ \la\bt_1\bt_1,a_0\ra a_0)=0$.
	This will yield the claim.

	We first calculate the term
	\begin{multline}
	\la\bt_1\bt_1,a_0\ra=\la\bt_1,\bt_1 a_0\ra=\frac{1}{4}\la\bt_1,\bt_1\ra\\
	=\frac{1}{4}\left(16\la\sg_1,\sg_1\ra-16\lm^2+\frac{63}{32}\lm+\frac{1}{2^7}(\mu-5)\right)
	=-\lm^2+\lm+\frac{1}{64}\left(\mu-1\right).
	\end{multline}
	Now $\al_1$ and $\bt_1$ have similar terms, so there is a lot of cancellation to produce
	\begin{multline}
	\al_1\al_1-\bt_1\bt_1=\frac{3}{8}\sigma_2^o-\frac{7}{2^9}(a_2+a_{-2})+(8\lm-2)\sg_1\\
	+\left(-\lm^2+\frac{1}{4}\lm-\frac{9}{2^8}\right)a_0+\left(-\frac{23}{8}\lm+\frac{75}{2^8}\right)(a_1+a_{-1})
	\end{multline}
	and thus
	\begin{multline}
	\alpha_1 \alpha_1-\beta_1 \beta_1+\la\beta_1\beta_1,a_0\ra a_0
	=(8\lm-2)\sg_1+\frac{3}{8}\sg_2^o-\frac{7}{2^9}(a_2+a_{-2})+\\
	+\left(-2\lm^2+\frac{5}{4}\lm+\frac{1}{64}\mu-\frac{13}{2^8}\right)a_0
	+\left(-\frac{23}{8}\lm+\frac{75}{2^8}\right)(a_1+a_{-1}).
	\end{multline}
	Now, finally,
	\begin{multline}
	0=a_0(\alpha_1 \alpha_1-\beta_1 \beta_1+\la\beta_1\beta_1,a_0\ra a_0)=\\
	= \frac{3}{8}a_0  \sigma_2^o + \left(-4\lm+\frac{19}{2^7}\right)\sg_1
	+ \left(4\lm^2-\frac{5}{8}\lm+\frac{1}{64}\mu+\frac{127}{2^{13}}\right)a_0 +\\
	+\left(-\frac{1}{16}\lm+\frac{19}{2^{13}}\right)(a_1+a_{-1})
	-\frac{7}{2^{14}}(a_2+a_{-2})-\frac{7}{2^8}\sg_2^e,
	\end{multline}
	and rearranging terms gives the result.
\end{proof}

\begin{lemma}
	\label{sigma sigma}
	We have
	\begin{multline*}\sg_1\sg_1=
	\frac{1}{3}\left[\left(-\frac{5}{4}\lm-\frac{13}{2^9}\right)\sg_1-\frac{7}{2^9}\sg_2^e+\frac{21}{2^{11}}\sg_2^o\right]+\\
	\frac{7}{3}\left[\left(\frac{1}{2}\lm^2-\frac{1}{2^7}\lm+\frac{1}{2^9}\mu-\frac{1}{2^{15}}\right)a_0+\left(\frac{7}{2^8}\lm-
	\frac{35}{2^{16}}\right)(a_1+a_{-1})+\frac{7}{2^{16}}(a_2+a_{-2})\right].
	\end{multline*}
\end{lemma}
\begin{proof}
	We use $\al_1$ as in the lemma previous.
	Recall also that $a_0 a_1=\sg_1+\frac{1}{32}(a_0+a_1)$.
	By Lemma \ref{prop-seress}, as $\al_1$ is a $0$-eigenvector, we have that
	$$(a_0 a_1)\al_1=a_0(a_1\al_1).$$
	Multiplying out the lefthand side gives
	\begin{multline}
	(a_0 a_1)\al_1
	=-4\sg_1\sg_1
	+\left(\left(3\lm-\frac{1}{4}\right)a_0+\frac{5}{16}a_1+\frac{7}{16}a_{-1}\right)\sg_1\\
	+\frac{1}{32}a_0\left(\left(3\lm-\frac{1}{8}\right) a_0+\left(3\lm+\frac{5}{16}\right)a_1+\frac{7}{16}a_{-1}\right)
	+\frac{7}{2^9}(a_1+a_1 a_{-1}),
	\end{multline}
	while on the right
	\begin{multline}
		a_0(a_1\al_1)=
		\left(\frac{3}{32}\lm-\frac{9}{2^9}\right)a_0+
		\left(-\frac{93}{32}\lm+\frac{279}{2^9}\right)a_0 a_1+
		\frac{7}{2^9}a_0 a_{-1} +
		(3\lm-1)a_0 \sg_1+\frac{7}{16}a_0\sg_2^o.
	\end{multline}
	The products $a_0\sg_1$, $a_1\sg_1$, $a_{-1}\sg_1$ are all known by Lemma \ref{times sigma}.
	Lemma \ref{second sigma} provides $a_0 \sg_2^o$.
	Thus we can rearrange to find the claim.
\end{proof}

\begin{lemma}
	\label{and a3}
	We have that
	\begin{multline*}
	a_3 = a_{-2}
		+ (2^8\lm-6)(a_{-1} - a_2)
		+ \left( \frac{2^{15}}{7}\lm^2 - \frac{2^8 3^2}{7}\lm + \frac{2^7}{7}\mu + 3\right)(a_0 - a_1)
	+ 2^5 7(\sg_2^o-\sg_2^e).
	\end{multline*}
\end{lemma}
\begin{proof}
	By Lemma \ref{sigma sigma}, we have an expression for $\sg_1\sg_1$,
	including in particular $a_{-2}$.
	Therefore we have an expression for $\overline{\sg_1\sg_1}$ including the term $a_3$.
	As $\overline{\sg_1\sg_1}=\bar\sg_1\bar\sg_1=\sg_1\sg_1$ by Lemma \ref{sigma one},
	from $\sg_1\sg_1 - \overline{\sg_1\sg_1} = 0$
	we can rearrange for the claimed expression for $a_3$.
\end{proof}

In the second half of this section,
we show that a certain set of size $8$ is sufficient
to span the universal $2$-generated Frobenius $\frak V(4,3)$-axial algebra $U$,
by finding the small number of products that are not yet computed.
The methods will be similar.

Set $R_0 = \QQ[\lm,\mu]$
and $W = \{a_{-2}, a_{-1}, a_0, a_1, a_2, \sg_1, \sg_2^e, \sg_2^o\}$.

\begin{lemma}
	\label{lem-T-invar}
	$R_0W$ is $T$-invariant.
\end{lemma}
\begin{proof}
	Lemmas \ref{sigma one} and \ref{sigma two} already established
	the $T$-invariance of $\sg_1$ and $\{\sg_2^e,\sg_2^o\}$.
	Then, $\tau_0$ preserves $\{a_{-2},a_{-1},a_0,a_1,a_2\}$,
	while the flip maps it to $\{a_{-1},a_0,a_1,a_2,a_3\}$.
	Lemma \ref{and a3} shows that $a_3$ is in $R_0W$, so the flip also preserves $R_0W$.
	Now recall that $T$ is generated by $\tau_0$ and the flip.

	In particular, $R_0W$ contains all $a_i$ for $i\in\ZZ$,
	and since for all $i$ the products $a_i\sg$ and $a_i\sg_2^e,a_i\sg_2^o$
	are contained in the $T$-orbits $(a_0\sg)^T,(a_0\sg_2^e)^T,(a_0\sg_2^o)^T$ respectively,
	we can deduce also these expressions
	from Lemmas \ref{times sigma}, \ref{lem a0 sg2} and \ref{sigma sigma}.
\end{proof}

\begin{lemma}
	\label{lem-ai-stable}
	$R_0W$ is $a_i$-stable.
\end{lemma}
\begin{proof}
	That is, $a_i w\in R_0W$ for all $w\in R_0W$ and $i\in \ZZ$.
	Without loss, we can assume that $w$ is in the spanning set $W$.
	Now all products $a_0w$ for $w\in W$ are known.
	Namely, $a_0a_0 = a_0,a_0a_{\pm1}=\sg_1+\frac{1}{32}(a_0+a_{\pm1}),a_0a_{\pm2}=\sg_2^e+\frac{1}{32}(a_0+a_{\pm2})$,
	$a_0\sg_1$ is recorded in Lemma \ref{times sigma},
	$a_0\sg_2^e$ in Lemma \ref{lem a0 sg2},
	and $a_0\sg_2^o$ in Lemma \ref{second sigma}.

	To complete the argument,
	as $T$ is transitive on the $a_i$,
	for all $i$ there exists $t_i$ such that $a_i = a_0^{t_i}$;
	then, for any $w$, $a_iw = {a_0}^{t_i}w = (a_0w^{t_i^{-1}})^{t_i}$,
	and as $w^{t_i^{-1}}\in R_0W$ by Lemma \ref{lem-T-invar},
	the righthandside of the expression is known.
	Here is the explicit list of formulas we use
	to express an arbitrary $a_iw$ in terms of $w'$ or $a_0w'$,
	for $i = -2,-1,1,2$ and $w,w'\in W$:
	\begin{gather}
		a_ia_i = a_i, \quad
		a_ia_{i\pm1} = \sg_1 + \frac{1}{32}(a_i+a_{i\pm1}), \\
		a_{2i}a_{2i\pm2} = \sg_2^e + \frac{1}{32}(a_{2i}+a_{2i\pm2}), \quad
		a_{2i+1}a_{2i+1\pm2} = \sg_2^o + \frac{1}{32}(a_{2i+1}+a_{2i+1\pm2}), \\
		a_{-2}a_1 = \overline{a_0a_3}, \quad
		a_{-2}a_2 = \overline{(\overline{a_0a_4})^{\tau_0}}, \text{where }
		a_4 = \overline{{a_3}^{\tau_0}}, \quad
		a_{-1}a_2 = (a_{-2}a_1)^{\tau_0}, \\
		a_1\sg_1 = \overline{a_0\sg_1}, \quad
		a_{-1}\sg_1 = (a_1\sg_1)^{\tau_0}, \quad
		a_2\sg_1 = \overline{a_{-1}\sg_1}, \quad
		a_{-2}\sg_1 = (a_2\sg_1)^{\tau_0}, \\
		a_1\sg_2^e = \overline{a_0\sg_2^o}, \quad
		a_{-1}\sg_2^e = (a_1\sg_2^e)^{\tau_0}, \quad
		a_2\sg_2^o = \overline{a_{-1}\sg_2^e}, \quad
		a_{-2}\sg_2^o = (a_2\sg_2^o)^{\tau_0}, \\
		a_1\sg_2^o = \overline{a_0\sg_2^e}, \quad
		a_{-1}\sg_2^o = (a_1\sg_2^o)^{\tau_0}, \quad
		a_2\sg_2^e = \overline{a_{-1}\sg_2^o}, \quad
		a_{-2}\sg_2^e = (a_2\sg_2^e)^{\tau_0}.
	\end{gather}
\end{proof}

\begin{lemma}
	\label{lem mult closed}
	$R_0W$ is closed under multiplication.
\end{lemma}
\begin{proof}
	After Lemma \ref{lem-ai-stable},
	we have that $a_iw\in R_0W$ for all $a_i$ and $w\in W$,
	so we only need to find products $vw$ for $v,w\in\{\sg_1,\sg_2^e,\sg_2^o\}$.
	We already have $\sg_1\sg_1$ from Lemma \ref{sg1 with sg1}.
	Also, by Lemma \ref{and a3},
	we may rewrite $\sg_2^o$ as a sum of $a_i$ and $\sg_1,\sg_2^e$.
	Hence, for any $w\in W$,
	an expression for $\sg_2^ow$ will follows from the other results.
	So we only now need to find $\sg_1\sg_2^e$ and $\sg_2^e\sg_2^e$.
	
	We make use of Lemma \ref{resurrection},
	the hypothesis of which hold with $R_0W$ in place of $B$
	by Lemma \ref{lem-ai-stable}.
	We relied on \cite{gap} in calculating these expressions.
	If $x = 16\sg_1\sg_2^e$
	then $b_\frac{1}{4} = -\al_1\bt_2 - x\in R_0W$, $b_0 = \al_1\al_2 -x\in R_0W$,
	and $y = x + b_\frac{1}{4}\in A_\frac{1}{4}$, $z = x + b_0\in A_0$,
	so Lemma \ref{resurrection} gives that
	$x = 4a(b_\frac{1}{4}-b_0)-b_\frac{1}{4}\in R_0W$,
	and the expression is
	\begin{equation}
		\begin{aligned}
			\label{ap-sg1sg2}
			\sg_1\sg_2^e & = 
			\frac{1}{3}\left(2^8\lm^3-\frac{27}{2}\lm^2+\lm\mu+\frac{17}{2^7}\lm-\frac{19}{2^9}\mu+\frac{19}{2^{15}}\right)a_0 \\
			& + \frac{1}{3}\left(14\lm^2-\frac{203}{2^8}\lm+\frac{665}{2^{16}}\right)\left(a_1+a_{-1}\right)
			+ \frac{1}{3}\left(\frac{7}{2^7}\lm-\frac{133}{2^{16}}\right)\left(a_2+a_{-2}\right) \\
			& + \frac{1}{3}\left(-2^4\cdot19\lm^2+\frac{41}{2}\lm+\frac{51}{16}\mu-\frac{197}{2^9}\right)\sg_1
			+ \frac{1}{3}\left(-\frac{17}{8}\lm+\frac{11}{2^8}\right)\sg_2^e
			+ \left(-\frac{7}{8}\lm+\frac{49}{2^{11}}\right)\sg_2^f.
		\end{aligned}
	\end{equation}
	Also, if $x = 16\sg_2^e\sg_2^e$
	then $b_\frac{1}{4} = -\al_2\bt_2 - x\in R_0W$, $b_0 = \al_2\al_2 -x\in R_0W$,
	so likewise Lemma \ref{resurrection} gives that $x\in R_0W$;
	we find that $\sg_2^e\sg_2^e$ is equal to
	\begin{equation}
		\begin{aligned}
			\label{ap-sg2sg2}
			 & \left(2^{19}\cdot5\lm^4-\frac{2^7\cdot6407}{3}\lm^3-2^7\cdot85\lm^2\mu+\frac{20303}{2}\lm^2+\frac{2329}{6}\lm\mu+\frac{3}{2}\mu^2-\frac{61409}{2^7\cdot3}\lm-\frac{5315}{2^9\cdot3}\mu+\frac{89069}{2^{15}\cdot3}\right)a_0 \\
			 & + \left(2^9\cdot7\lm^3-\frac{791}{3}\lm^2+\frac{2317}{2^7\cdot3}\lm-\frac{8645}{2^{16}\cdot3}\right)\left(a_1+a_{-1}\right)
			 + \left(-49\lm^2+\frac{343}{2^6\cdot3}\lm+\frac{21}{2^8}\mu-\frac{3563}{2^{16}\cdot3}\right)\left(a_2+a_{-2}\right) \\
			 & + \left(-2^{20}\cdot3\lm^4+2^{14}\cdot45\lm^3-2^{12}\cdot3\lm^2\mu-\frac{2^4\cdot7523}{3}\lm^2+2^6\cdot7\lm\mu+\frac{4819}{6}\lm-\frac{65}{16}\mu-\frac{65}{12}\right)\sg_1 \\
			 & + \left(-2^{14}\cdot3\lm^3-2^4\cdot99\lm^2-2^6\cdot3\lm\mu+\frac{2837}{24}\lm+\frac{47}{16}\mu-\frac{4079}{2^{10}\cdot3}\right)\sg_2^e
			 + \left(-2^5\cdot21\lm^2+\frac{49}{2}\lm-\frac{455}{2^{11}}\right)\sg_2^f.
		\end{aligned}
	\end{equation}

	We have now computed all products in $R_0W$.
\end{proof}

\begin{lemma}
	The bilinear form $\la,\ra$ is $R_0$-valued on $R_0W$.
\end{lemma}
\begin{proof}
	After Lemma \ref{no parity},
	we have that $\la x,y\ra^t = \la x,y\ra$ for all $t\in T$.
	By linearity, it suffices to consider the values of the form on monomials;
	furthermore, since the form is associating,
	any expression $\la x,y\ra$ can be rewritten as a sum of expressions
	$\la a_i,w\ra$ for $w\in W$.
	Now we may assume that $i=0$.
	By definition, $\la a_0,a_{\pm1}\ra = \lm$ and $\la a_0,a_{\pm2}\ra = \mu$.
	By Lemma \ref{with sg1}, $\la a_0,\sg_1\ra$ is known and lies in $R_0$.
	With 
	\begin{align}
		\la a_0, \sg_2^e \ra
		&= \la a_0, a_0 a_2 \ra-\frac{1}{2^5}\left(\la a_0,a_0 \ra + \la a_0,a_2 \ra\right) 
		= \frac{1}{2^5}(31\mu-1),\\
		\la a_0, \sg_2^o \ra &= \la a_0, a_{-1} a_1 \ra - \frac{1}{2^5}(\la a_0, a_-1 \ra + \la a_0, a_1 \ra ) \\
		&= \la \sg_1 +\frac{1}{2^5}(a_0+a_1), a_{-1} \ra - \frac{1}{2^4}\lm \\
		&= \frac{1}{2^5}\left(30\lm+\mu-1\right)
	\end{align}
	we have made all the necessary calculations.
	For additional completeness, we also present
	\begin{align}
		& \begin{aligned}
		\nu_3 = {} & \la a_{-2},a_1\ra = \la a_{-1},a_2\ra = \la a_0,a_3\ra \\
			= {} & -\frac{1}{7}(2^{15}\lm^3-2^{12}\cdot3^2\lm^2+2^7\cdot15\lm\mu+2169\lm+33\mu-33), \\
		\end{aligned} \\
		& \begin{aligned}
		\nu_4 = {} & \la a_{-2},a_2\ra = \la a_0,a_4\ra
			= \frac{1}{7}(2^{23}\lm^4-2^{15}\cdot293\lm^3+2^{16}\cdot7\lm^2\mu\\ & + 2^{12}\cdot189\lm^2 -2^7\cdot5\lm\mu-2^7\mu^2-2^7\cdot155\lm-21\mu+156).
		\end{aligned}
	\end{align}
	It follows that the ring $R$,
	the extension of $\QQ$ by all possible values of the form $\la,\ra$,
	is equal to $R_0 = \QQ[\lm,\mu]$.
\end{proof}

\begin{theorem}
	\label{thm-alg-determined}
	The universal $2$-generated Frobenius $\frak F$-axial algebra $U$
	is equal to $R_0W$,
	and the products and values of the form on $U$ have been determined.
	\qed
\end{theorem}

\section{Sakuma's theorem}
\label{sec-sakuma}

In Theorem \ref{thm-alg-determined}, we found that $U = R_0W$
is the universal $2$-generated Frobenius $\frak V(4,3)$-axial algebra.
We can certainly extract additional information
to describe both $R = R_0$ and $U$ in more detail.

We take some relations arising from the associativity of the bilinear form.
\begin{lemma}
	\label{lem-assoc-pol}
	The following expressions must be zero:
	\begin{align}
		& \begin{aligned}
		p_1 & = \la a_{-1} a_{-2}, a_1 \ra - \la a_{-1}, a_{-2} a_1 \ra \\
			& = \lm^4 -\frac{71}{2^6}\lm^3 + \frac{5}{2^8}\lm^2\mu +\frac{45}{2^9}\lm^2 +\frac{139}{2^{15}}\lm\mu +\frac{1}{2^{14}}\mu^2 -\frac{75}{2^{15}}\lm -\frac{167}{2^{21}}\mu +\frac{39}{2^{21}}
		\end{aligned} \\
		& \begin{aligned}
		p_2 & = \la a_{-2}, a_1 \ra - \la a_{-2}, a_{-2} a_1 \ra \\
			& = \lm^5 -\frac{577}{2^9}\lm^4 +\frac{25}{2^9}\lm^3\mu +\frac{1347}{2^{14}}\lm^3 -\frac{389}{2^{17}}\lm^2\mu +\frac{23}{2^{17}}\lm\mu^2 -\frac{105}{2^{16}}\lm^2 +\frac{5183}{2^{24}}\lm\mu \\
			& + \frac{87}{2^{24}}\mu^2 -\frac{63}{2^{24}}\lm -\frac{2901}{2^{29}}\mu +\frac{117}{2^{29}}
		\end{aligned}
    \end{align}
\end{lemma}
\begin{proof}
	That $p_1,p_2$ are zero is immediate from the associativity of the form.
	We carried out the evaluation of the form in \cite{gap}.
\end{proof}

Set $\lm',\mu'$ to be two formal variables,
$R' = \QQ[\lm',\mu']$ their polynomial ring over $\QQ$,
and $p_1',p_2'\in R'$ the images of $p_1,p_2$ under substitution of $\lm',\mu'$ for $\lm,\mu$.
(Then $R'$ differs from $R = R_0 = \QQ[\lm,\mu]$
in that $\lm$ and $\mu$ satisfy additional relations
enforced by the axioms of Frobenius axial algebras,
whence $R$ is a quotient of $R'$.)
\begin{lemma}
	\label{lem-ideal}
	The ideal $I = (p_1',p_2')\normal R'$ is radical;
	the quotient $R'/I$ is $9$-dimensional,
	corresponding to the following points $(\lm',\mu')$ in $\QQ^2$:
	$$ \left(1,1\right),
	\left(0,1\right),
	\left(\frac{1}{8},1\right),
	\left(\frac{1}{64},\frac{1}{64}\right),
	\left(\frac{13}{2^8},\frac{13}{2^8}\right),
	\left(\frac{1}{32},0\right),
	\left(\frac{1}{64},\frac{1}{8}\right),
	\left(\frac{3}{2^7},\frac{3}{2^7}\right),
	\left(\frac{5}{2^8},\frac{13}{2^8}\right).
	$$
	Hence $R'/I\cong\QQ^9$.
\end{lemma}
\begin{proof}
	Using \cite{magma},
	the following functions give the above results:
    \begin{verbatim}
	 	R_0<lm,mu> := PolynomialRing(Rationals(),2);
	 	p_1 := lm^4-71/64*lm^3+5/256*lm^2*mu+45/512*lm^2+139/32768*lm*mu+...
	 	p_2 := lm^5-577/512*lm^4+25/512*lm^3*mu+1347/16384*lm^3-...
	 	I := Ideal([p_1,p_2]);
	 	IsRadical(I);
	 	> true
	 	Q := quo< R_0 | [p_1,p_2] >;
	 	Dimension(Q);
	 	> 9
	 	Sch := Scheme(Spec(R_0), [p_1,p_2]);
	 	Dimension(Sch);
	 	> 0
	 	Points(Sch);
	 	> {@ (0, 1), (1/64, 1/64), (1/64, 1/8), (5/256, 13/256), (3/128, 3/128),
	 		(1/32,0), (13/256, 13/256), (1/8, 1), (1, 1) @}
		\end{verbatim}
	We are grateful to Miles Reid for instructing us in making this calculation.
\end{proof}

Now we want to define particular quotients of $U$
corresponding to the above solutions for $\lm'$ and $\mu'$.
Set $U_i = M(\{\lm-\lm_i,\mu-\mu_i\},\emptyset)$
as in Section \ref{sec-uni-alg},
for some pair $(\lm_i,\mu_i)$ from Lemma \ref{lem-ideal}.
We will call this quotient the evaluation of $U$ for $\lm=\lm_i,\mu=\mu_i$
and denote the associated mapping by $U\to U_i$ by $\phi_i$.
Clearly either $U_i$ is trivial or its ring is $\QQ$.

\begin{lemma}
	\label{lem ns quots}
	Every Norton-Sakuma algebra $n_iX_i$
	is a quotient of the suitable evaluation $U_i$ of $U$.
	The correspondence is recorded below,
	as are the dimensions of the $n_iX_i$.
	\renewcommand{\arraystretch}{1.75}
	\begin{center}
	\begin{tabular}{|r||c|c|c|c|c|c|c|c|c|c|}
		\hline
			$(\lm_i,\mu_i)$ &
			$\left(1,1\right)$ &
			$\left(0,1\right)$ &
			$\left(\frac{1}{8},1\right)$ &
			$\left(\frac{1}{64},\frac{1}{64}\right)$ &
			$\left(\frac{13}{2^8},\frac{13}{2^8}\right)$ &
			$\left(\frac{1}{32},0\right)$ &
			$\left(\frac{1}{64},\frac{1}{8}\right)$ &
			$\left(\frac{3}{2^7},\frac{3}{2^7}\right)$ &
			$\left(\frac{5}{2^8},\frac{13}{2^8}\right)$ \\
		\hline
			$n_iX_i$ &
			$1A$ &
			$2B$ &
			$2A$ &
			$3C$ &
			$3A$ &
			$4A$ &
			$4B$ &
			$5A$ &
			$6A$ \\
		\hline
			$\op{dim}$ &
			$1$ &
			$2$ &
			$3$ &
			$3$ &
			$4$ &
			$5$ &
			$5$ &
			$6$ &
			$8$ \\
		\hline
	\end{tabular}
	\end{center}
\end{lemma}
\begin{proof}
	Firstly, the Norton-Sakuma algebras are known
	to be Majorana algebras \cite{ipss},
	and hence they are also Frobenius $\frak V(4,3)$-axial algebras.
	Since they are $2$-generated,
	we know that they are quotients of our universal object $U$.

	Secondly, after rescaling the form to match \cite{conway},
	it is easy to see that the values of $\lm$ and $\mu$ in $U_i$
	are the same as those given for the $2$-generated subalgebras $n_iX_i$ on page 529.
	(There is no need to rescale to match \cite{ipss}.)
	This means exactly that $n_iX_i$ is a quotient of the algebra $U_i$.
\end{proof}

In particular, since every $U_i$ has a nontrivial quotient,
the $U_i$ themselves are nontrivial.

\begin{corollary}
	\label{cor q9}
	The ring $R = R_0 = \QQ[\lm,\mu]$ is isomorphic to $\QQ^9$.
\end{corollary}
\begin{proof}
	Evidently $R$ is a quotient of $R'/I$,
	the polynomial ring in two formal variables $\lm',\mu'$
	quotiented by the ideal $I$ of Lemma \ref{lem-ideal}.
	Hence $R$ is a quotient of $\QQ^9$.
	On the other hand, since all nine Norton-Sakuma algebras
	are realised as quotients of $U$,
	every point $(\lm_i,\mu_i)$ in $R'/I$ is realised by $R$,
	and therefore $R$ is $9$-dimensional,
	that is, $R \cong \QQ^9$.
\end{proof}

\begin{corollary}
	\label{cor u direct}
	The algebra $U$ is the direct sum of $U_i$.
\end{corollary}
\begin{proof}
	By Corollary \ref{cor q9},
	$R$ has an identity, denoted $1$,
	which decomposes as a direct sum of nine primitive pairwise annihilating idempotents $e_1,\dotsc,e_9$.
	Observe that multiplying by $e_i$
	corresponds to an evaluation mapping $U\to U_i$
	(where we may choose the labeling of the $e_i$
	so that indeed $U_i$ is the image of $\ad(e_i)$).
	Then set $R_i = Re_i$, so that $U_i = R_iW$,
	and as $1 = \sum_i e_i$,
	we have $R = \sum_i R_i$ and $U=\sum_i U_i$.
	Now $R = \oplus_iR_i$ and each $R_i\cong\QQ$.
	Suppose that $x\in U_i\cap(\sum_{j\neq i}U_j)$.
	Then $x = 1x = \sum_k e_kx$;
	as $x\in U_i$, $e_kx = 0$ for all $k\neq i$,
	and as $x\in \sum_{j\neq i}U_j$, also $e_ix=0$.
	Therefore $x = \sum_k 0 = 0$,
	and the intersection is trivial.
	Therefore $U = \oplus_i U_i$.
\end{proof}

\begin{lemma}
	\label{lem exactly ns}
	Each quotient algebra $U_i$ is isomorphic
	to the corresponding Norton-Sakuma algebra $n_iX_i$.
\end{lemma}
\begin{proof}
	Section \ref{sec-cal} gave us enough information to construct a new algebra, $A$,
	which will help us in this proof.
	Set $W' = \{a_{-2}',a_{-1}',a_0',a_1',a_2',\sg_1',{\sg_2^e}',{\sg_2^o}'\}$,
	and $A$ to be the free $R'$-module with basis $W'$
	and algebra products defined by our results in Section \ref{sec-cal}
	(replacing, in all instances, coefficients in $R$ with coefficients in $R'$
	and elements of $W$ with elements of $W'$).
	Clearly $U$ is a quotient of $A$.
	Furthermore, just as we have a group $T$ of automorphisms
	acting on $U$, generated by $\tau_0$ and the flip,
	we have a group $T'$ acting on $A$
	which is generated by the two automorphisms
	\begin{gather}
		\label{eq t'}
		a_i'\mapsto a_{-i}',\quad \sg_1',{\sg_2^e}',{\sg_2^o}'\text{ fixed; } \\
		a_i'\mapsto a_{1-i}',\quad \sg_1'\mapsto\sg_1',{\sg_2^e}'\mapsto{\sg_2^o}',{\sg_2^o}'\mapsto{\sg_2^e}'.
	\end{gather}

	To verify that $U_i$ indeed is its corresponding Norton-Sakuma algebra,
	we do the following.
	We first apply the evaluation mapping $(\lm',\mu')\mapsto(\lm_i,\mu_i)$
	to the structure constants of $A$,
	to produce an $8$-dimensional algebra $A_i$ over $\QQ$.
	Note $U_i$ is a quotient of $A_i$.
	Set $\psi_i$ to be the map $A\mapsto A_i$ and
	\begin{equation}
		I_i = ( \psi_i(x)\psi_i(y) - (\psi_i(x)^{\psi_i(t)}\psi_i(y)^{\psi_i(t)})^{\psi_i(t^{-1})}\mid t\in T',x,y\in W' ),
	\end{equation}
	where we also extend the action of $\psi_i$ to $T'$
	by the same evaluation of entries of the matrix representation of $T'$ on $A$.
	We know that $I_i$ must be in the kernel of the map $A_i\mapsto U_i$.
	We used \cite{gap} to calculate that the dimension of $I_i$ over $\QQ$ is
	$7,6,5,5,4,3,3,2,0$ respectively,
	in the same order as recorded in Lemma \ref{lem ns quots}.
	Therefore we find that the dimension of $U_i$
	is at most $1,2,3,3,4,5,5,6,8$ respectively.
	But these are exactly the dimensions of the Norton-Sakuma algebras $n_iX_i$,
	and hence the upper bound meets the lower bound
	to confirm that indeed the $U_i$ are the Norton-Sakuma algebras.
\end{proof}

\begin{theorem}
	\label{thm-sakuma}
	The universal $2$-generated Frobenius $\frak V(4,3)$-axial algebra $U$
  is isomorphic to the direct sum of the Norton-Sakuma algebras over $\QQ$.
	\qed
\end{theorem}

\end{document}